\documentclass[11pt,twoside, leqno]{article}

\usepackage{mathrsfs}
\usepackage{amssymb}
\usepackage{amsmath}
\usepackage{amsthm}
\usepackage{amsfonts}
\usepackage{extarrows}
\usepackage{xy}
\xyoption{all}

\usepackage{color}
\usepackage[colorlinks, linkcolor=red, anchorcolor=green, citecolor=blue]{hyperref}
\usepackage[pagewise]{lineno}

\allowdisplaybreaks

\renewcommand\hat{\widehat}
\renewcommand\tilde{\widetilde}

\def\r{\right}
\def\lf{\left}

\def\bint{{\ifinner\rlap{\bf\kern.30em--}
\int\else\rlap{\bf\kern.35em--}\int\fi}\ignorespaces}

\def\sbint{{\ifinner\rlap{\bf\kern.32em--}
\hspace{0.078cm}\int\else\rlap{\bf\kern.45em--}\int\fi}\ignorespaces}

\newtheorem{theorem}{Theorem}[section]
\newtheorem{lemma}[theorem]{Lemma}
\newtheorem{corollary}[theorem]{Corollary}

\theoremstyle{definition}
\newtheorem{remark}[theorem]{Remark}
\newtheorem{definition}[theorem]{Definition}

\newtheorem {conjecture} [theorem] {Conjecture}

\newtheorem*{lpthm}{\noindent\bf Lions-Peetre Iterative Theorem}
\newtheorem*{hthm}{\noindent\bf Hunt Theorem}
\newtheorem*{holthm}{\noindent\bf Holmstedt Theorem}

\numberwithin{equation}{section}

\numberwithin{equation}{section}

\numberwithin{equation}{section}

\begin{document}
\UseRawInputEncoding
\arraycolsep=1pt
\title{\Large\bf Peetre conjecture on real interpolation spaces of Besov spaces and Grid K functional
\footnotetext{\hspace{-0.35cm} {\it 2010 Mathematics Subject
Classification}. {46B70, 46E35, 42B35.}
\endgraf{\it Key words and phrases.}
Real interpolation; Besov spaces; K functional;
Wavelet basis and grid; nonlinear functional and nonlinear topology.
\endgraf $^\ast$\,Corresponding author
}}
\author{ Qixiang Yang,  Haibo Yang, Bin Zou and  Jianxun He$^{\ast}$}
\date{  }
\maketitle

\vspace{-0.8cm}

\begin{center}
\begin{minipage}{13cm}\small
{\noindent{\bf Abstract}
In this paper, Peetre's conjecture about the real interpolation space of Besov space
{\bf is solved completely }
by using the classification of vertices of cuboids defined by {\bf wavelet coefficients and wavelet's grid structure}.
Littlewood-Paley analysis provides only a decomposition of the function on the ring.
We extend Lorentz's rearrangement function and Hunt's Marcinkiewicz interpolation theorem to more general cases.
We use the method of calculating the topological quantity of the grid
to replace the traditional methods of data classification
such as gradient descent method and distributed algorithm.

We developed a series of new techniques to solve this longstanding open problem.
These skills make up for the deficiency of Lions-Peetre iterative theorem in dealing with strong nonlinearity.
Using the properties of wavelet basis, a series of {\bf functional nonlinearities} are studied.
Using the lattice property of wavelet, we study the lattice topology.
By three kinds of {\bf topology nonlinearities}, we give the specific wavelet expression of K functional.

}

\end{minipage}
\end{center}


\section{Background and main theorem}\label{s1}
\hskip\parindent

As early as 1911, Schur \cite{Schur} considered the interpolation properties of operators.
In 1926, M. Riesz \cite{R} proved the first version of the Riesz-Thorin theorem.
In 1948, Salem-Zygmund \cite{SZ} and Thorin  \cite{Thorin}
appeared the study of interpolation on $H^{p}$ spaces.
The study of interpolation spaces greatly enriches function space and operator theory and
has important applications in differential equations.
The study of interpolation space has greatly aroused the attention of mathematicians.
J. Bergh -J. L\"ofstr\"om wrote on page 170 of their book \cite{BL} that
Notices A.M.S introduces 12 open problems of interpolation proposed by 12 mathematicians
under the heading ``Problems in interpolation of operations and applications I-II" in
Notices Amer. Math. Soc. See \cite{PCFGLM, SHSTBR}.
In order of name, these 12 mathematicians are
J. Peetre, W. Connett, J. Fourier, J. E. Gilbert, J. L. Lions, Benjamin Muckenhoupt,
E. M. Stein, Richard A. Hunt, Robert C. Sharpley, A. Torchinsky, Colin Bennett and N. M. Riviere.
These problems are aimed at {\bf the special phenomena in the process of interpolation}
and have received the attention they deserve and some have been resolved.
The theory of interpolation spaces effectively reveals the internal structure of functions and operators,
and has played a great role in promoting the development of harmonic analysis for a long time.
See \cite{AK, AKMNP5,BL, BC, BCT, B, CZ, C, CD, DP,DY, D, FRS, FS, HS, HP, Litt, LoP,
Peetre67, PS, SaZ, ST,  SZ, Triebel,TriebelB,YDC,YCP}.

For Banach spaces $A_0$ and $A_1$, in \cite{PCFGLM},
Peetre wrote it is well known that
$$(A_0,A_1)_{\theta,1}\subset [A_0,A_1]_{\theta} \subset (A_0,A_1)_{\theta,\infty}.$$ 
Hence real interpolation spaces $(A_0,A_1)_{\theta,r}$ have a more careful non-linear structure than complex interpolation spaces $[A_0,A_1]_{\theta}$.
In the 1950s, Lorentz \cite{Lorentz} and Lorentz-Zeller \cite{LZ}
considered distribution functions and
non-increasing rearrangement functions and introduced Lorentz spaces.
Hunt \cite{Hunt1, Hunt2} established the relation between Lorentz spaces and real interpolation for Lebesgue spaces
and generalized Marcinkiewicz theorem to Lorentz spaces.
The study of real interpolation spaces greatly expands the Lebesgue spaces
and improves the study of continuity of operators.
And this eventually forms the Lorentz space theory of real interpolation spaces.
In the learning algorithm, the cost function with regularization term is just the K-functional of real interpolation space,
so the theoretical study of K-functional provides a theoretical support for the learning algorithm.

The above describes the phenomena related to interpolation.
{\bf The concrete expression of interpolation space is the essence of interpolation}.
But even for the real interpolation space of Besov space,
because of the complex non-linear structure,
a very few index spaces have the concrete expression.
Besov spaces, developed in the 1970s, are very important
in harmonic analysis, operator theory, approximation theory and differential equations.
Therefore, Peetre proposed on page 110 of \cite{Peetre} to find out
the specific expression form of the real interpolation space
$(B^{s_0, q_0}_{p_0}, B^{s_1, q_1}_{p_1})_{\theta, r}$ when $r\neq q$.
See conjecture \ref{con:1} and Peetre's book \cite{Peetre67book} in 1967.
In 1974, Cwikel \cite{C} proved that
the Lions-Peetre formula
have no reasonable generalization for any $r\neq q$.
{\bf Peetre explained the difficulty of the problem,
such interpolation space will run out of the Besov type spaces,
and the existing techniques can not control different indices at the same time.}
In operator theory and nonlinear partial differential equations,
the most difficult problem to deal with is the nonlinear structure of the functions.
In the process of solving Peetre's conjecture,
we find that the problem has two types of nonlinearity:
{\bf functional nonlinearity and topological nonlinearity}.

When $p_0=p_1=p$ and $r=q$, $(\dot{B}_{p}^{s_{0},q_{0}}, \dot{B}_{p}^{s_{_{1}},q_{1}})_{\theta, r}$ are still Besov spaces.
D.C. Yang has considered such cases for metric spaces, see \cite{BL,TriebelB,YDC}.
Q.X. Yang has been paying attention to real interpolation spaces since 2000,
and several European mathematicians have inquired about Yang-Cheng-Peng's results  in various ways.
See \cite{BC, BCT, CD, HS, YCP}.
For decades, there have been many works focusing on Besov spaces and interpolation spaces and their applications,
see \cite{AK,AKMNP5, BaC, BL, BC, BCT, CL, CD,DP,DY, HS, HP, LiuYY, ST, Triebel,TriebelB,YDC,YCP, ZYY} and the references theirin.
But due to the lack of appropriate techniques and tools, Peetre conjecture has not made much progress for a long time.
Under the constraints $s_0=s_1, q_0=q_1$,
by using wavelets to process,
Lou-Yang-He-He \cite{LYHH} made progress for $p_0\neq p_1$ recently.
In addition, Besoy-Haroske-Triebel \cite{BHT}
are also concerned about this longstanding open problem.
They  have been able to describe
$(\dot{B}_{p_0}^{s-\alpha / p_0, p_{0}}, \dot{B}_{p_1}^{s- \alpha / p_1, p_{1}})_{\theta, r}$
where $0<p_0<p_1<\infty, 0<\theta<1, -\infty<\alpha<\infty, 0<r\leq \infty$ and $s\in \mathbb{R}$.
They did it via its wavelet representation.
More specifically, they used  trace theorem and \cite{BCT, YCP}
to consider $(\dot{B}_{p_0}^{s_{0},q_{0}}, \dot{B}_{p_1}^{s_{_{1}},q_{1}})_{\theta, r}$ for the cases
where $p_0=q_0$, $p_1=q_1$  and $s_0-s_1= \frac{\alpha}{p_0}-\frac{\alpha}{p_1}$.

{\bf After the first author's 20 years study and the cooperation of all the authors},
the nonlinear nature of interpolation has finally been fully discovered:
cuboid functional and vertex functional, nonlinear topology defined by power spaces
and topology compatibility for intermediate spaces.
Peetre's conjecture is completely solved by using both the basis property and grid property of wavelets.
After more than 20 years of efforts,
we finally find out all the expression of relative K functional
based on grid topology and the K functional on layer grid.
{\bf Roughly speaking}, for a function defined by some norm on a layered grid $\Gamma_j$,
we obtain some kind of function by some nonlinear composition.
And then we sum these functions over the main grid $\mathbb{Z}$.
And finally, the sum of this function, we take the nonlinear composition again,
and we end up with a K functional on full grid $\Lambda$.
That is to say,
we have the following wavelet characterization of the real interpolation space for all Besov spaces:

\begin{theorem}\label{th:main}
Let $s_0,s_1\in \mathbb{R}$,
$ 0<\theta<1$ and $0<p_0,p_1,q_0,q_1, r\leq \infty$.  We have
$$f\in (B^{s_0, q_0}_{p_0}, B^{s_1, q_1}_{p_1})_{\theta, r} \Leftrightarrow
\lf\{\int^{\infty}_{0} [ t^{-\theta} K(t,f, B^{s_0, q_0}_{p_0}, B^{s_1,q_1}_{p_1})]^{r} \frac{dt}{t}\r\}^{\frac{1}{r}}<\infty, $$
where the wavelet representation of K functional $K(t,f, B^{s_0, q_0}_{p_0}, B^{s_1,q_1}_{p_1})$
are nonlinearly given in four cases as follows:

{\rm (i)} For $0<p_0= p_1=p\leq \infty$ and $0<q_0, q_1\leq \infty$,
the concrete wavelet representation of $K(t,f, B^{s_0, q_0}_{p_0}, B^{s_1,q_1}_{p_1})$ is defined
in the following Corollaries \ref{cor:7.7} and \ref{cor:7.10}
and in the following equation \eqref{eq:973}
of Theorem \ref{th:7.11} by wavelet coefficients.

{\rm (ii)} For $0<p_0< p_1\leq \infty$ and $0<q_0=q_1\leq \infty$,
we obtain the wavelet characterization of the K functional
$K(t,f, B^{s_0, q_0}_{p_0}, B^{s_1,q_1}_{p_1})$ using two methods.
For the first method, the wavelet characterization of
$K(t,f, B^{s_0, q_0}_{p_0}, B^{s_1,q_1}_{p_1})$ is defined
in the following Theorems \ref{th:8.1}, \ref{th:8.2} and \ref{th:8.3} respectively
by cuboid functional and K functional on grid.
The second method makes use of the vertex K functional and K functional on layer grid,
and the wavelet characterization of the K functional $K(t,f, B^{s_0, q_0}_{p_0}, B^{s_1,q_1}_{p_1})$
is defined in the following Theorem \ref{th:10.1}.

{\rm (iii)} For $0<p_1< p_0\leq \infty$ and $0<q_0=q_1\leq \infty$,
by using vertex K functional and K functional on layer grid,
the concrete wavelet representation of $K(t,f, B^{s_0, q_0}_{p_0}, B^{s_1,q_1}_{p_1})$ is defined by wavelet coefficients
in the following Corollary \ref{cor:10.1}.

{\rm (iv)} For $0<p_0 \neq  p_1\leq \infty$ and $0<q_0\neq q_1< \infty$, by using vertex K functional and exponential spaces,
the concrete wavelet representation of $K(t,f, B^{s_0, q_0}_{p_0}, B^{s_1,q_1}_{p_1})$ is defined by wavelet coefficients
in the following Theorem \ref{th:11.4}.

\end{theorem}

{\bf For ease of reading, let's first explain our K functional expression}.
For $p_0=p_1=p$,
we first take the $L^{p}$ norm on the layer grid to get the sequence on the main grid,
and then take the K functional for this sequence on the main grid.
For $p_0\neq p_1$ and $q_0=q_1$,
first, we take K functional on the layer grid,
and then use the compatibility between the main grid and the full grid
to get K functional on the full grid.
For $p_0\neq p_1$ and $q_0\neq q_1$,
we first take K functional on the layer grid,
and then use power spaces to obtain K functional of $(Y_0^j, Y_1^j)$.
Then, the K functional of $(X_0,X_1)$ is obtained by using the compatibility of the main grid and the full grid to the intermediate space.
Finally, the K functional on the full grid is obtained again by using the power spaces.

\begin{remark}

{\bf For ease of reading, let's introduce our proof idea also.}

Wavelets have both the basis properties and the fully discrete lattice properties.
This allows us to discuss functional structures by using the properties of basis
and the topology of lattice points by using the properties of fully discrete lattice points.

{\bf Firstly}, we use wavelet to get wavelet K functional.
Then we obtain the cuboid functional by analyzing the wavelet coefficient characteristics that can approximate the wavelet functional.
{\bf Secondly},
we establish K functionals on the main grid and layer grid,
which provides support for the construction of K functions on the full grid
by using the lattice topological structure starting from the functions on the layer grid.
{\bf Thirdly}, by K functionals on the main grid and layer grid and by cuboid functional,
we can deal with the specific wavelet expressions of the K functional for $p_0=p_1$ or $q_0=q_1$.

{\bf Fourthly}, we consider vertex K functional.
This transforms the calculation problem of functional into the classification problem of cuboid vertices.
{\bf Fifthly}, using vertex functional,
we obtain again the expression of K functional when $q_0=q_1$
through the grid topology compatibility.
{\bf Sixthly and lastly}, we consider the cases where $p_0\neq p_1$ and $q_0\neq q_1$.
We first take K functional on the layer grid,
and then use power spaces to obtain K functional of $(Y_0^j, Y_1^j)$.
Then, the K functional of $(X_0,X_1)$ is obtained by using the compatibility of the main grid and the full grid to the intermediate space.
Finally, the K functional on the full grid is obtained again by using the power spaces.

\end{remark}

\begin{remark}
Peetre's conjecture is posed for the indices $1\leq p_0, p_1\leq \infty$.
By a series of complex new wavelet skills,
{\bf our main Theorem \ref{th:main} not only solves completely the Peetre's conjecture on page 110 of \cite{Peetre},
but also considers more general indices than Peetre's conjecture.}
\end{remark}

\begin{remark}
In 1950, Lorentz \cite{Lorentz} introduced rearrangement function and generalized Lebesgue space.
In 1964, Hunt \cite{Hunt1, Hunt2} established the relationship
between Lorentz spaces and real interpolation spaces about Lebesgue spaces.
The space here is defined by taking specific nonlinear operations on the functions on the layer grid,
using the compatibility of the lattice topology and the nonlinear relationship of the main grid.
Our main Theorem extend the results of Lorentz \cite{Lorentz} and Hunt \cite{Hunt1, Hunt2}
for Lebesgue spaces  to which for Besov spaces
via the basis properties of wavelet and the fully discrete lattice properties of wavelet.
Besov spaces and Triebel-Lizorkin spaces are systematically studied in
Peetre \cite{Peetre} and Triebel \cite{TriebelB}.
Our technique of defining function spaces from topologies on fully discrete grids
also provides a basis for systematically generalizing these two types of spaces.
\end{remark}

\begin{remark}

In this paper, the whole K functional of real interpolation spaces are clarified through wavelets and power spaces,
so as to obtain the specific wavelet expression of K functional of general index of real interpolation of Besov spaces.
The real interpolation spaces of Besov space of general index are obtained.
For $r\neq q$, our real interpolation spaces are all new except the spaces in Lou-Yang-He-He \cite{LYHH} and Besoy-Haroke-Triebel \cite{BHT}.

Based on our results,
the inclusion relationship between Lorentz space and the real interpolation spaces of Besov spaces
can be conveniently considered.
As a direct application, we can generalize the real interpolation inequalities of Bahouri-Cohen \cite{BaC} and  Chamorro-Lemari\'e \cite{CL}
to more general indices.
Ou-Wang \cite{OW} considered Stein's conjecture.
Further, we can try to extend the estimation of $Ef$ in Ou-Wang's theorem to Lorentz space by real interpolation idea.
\end{remark}

The rest of this article is structured as follows:

\vspace{0.2cm}
From Section \ref{sec:2} to Section \ref{sec:4},
we introduce some preparation knowledge.
In Section \ref{sec:2}, we present some preliminaries on interpolation spaces.
First, we recall the definition of K functional and some relative basic properties which will be used later.
Then we present some basic knowledge on Lorentz spaces.
In Section \ref{sec:3}, we introduce first the definition of Besov spaces and Besov-Lorentz spaces,
then we present what's Peetre's conjecture on the real interpolation of Besov spaces.
In Section \ref{sec:4}, we recall first some preliminaries on wavelets and function characterization.
Wavelets have both the basis properties and the fully discrete lattice properties.
This allows us to discuss functional structures by using the properties of basis
and the topology of lattice points by using the properties of fully discrete lattice points.
Then we introduce the progress of Peetre conjecture before our work.

\vspace{0.2cm}
In Section \ref{sec:6}, we consider the preliminary functional structure of K functional.
First we transform continuous K functionals into wavelet K functionals in Theorem \ref{th:6.1}.
Then we analyze the characteristics of the wavelet coefficients approximating the wavelet functional
and get the cuboid functional.
See Theorem \ref{th:5.3}.
These results give Theorem 6 in \cite{DP} in a more explicit way.
Hence wavelet basis are unconditional basis for real interpolation spaces.
See Theorem \ref{th:absolute}.
The content of Section 5 provides us with functional tools to get the K functional
for the cases $p_0=p_1$ or $q_0=q_1$,
and prepares us for further vertex functional.

\vspace{0.2cm}
In Section \ref{sec:5},
we build K functional on a simple grid.
In Sections \ref{sec:5.1}, we consider the case for main grid $\mathbb{N}$.
In Sections \ref{sec:5.2}, we consider the case for layer grid $\Gamma_{j}$.
The results in Sections \ref{sec:5.1} and \ref{sec:5.2} are corresponding to the generalization of the continuous case,
we only state theorems without giving proofs.
In Section \ref{sec:8.1} we study the K functionals when the function is restricted to
a single layer and the classification on the corresponding index set.
These provide necessary preparatory knowledge for the computation of the complex K functional given by wavelets later.

\vspace{0.2cm}
In Section \ref{sec:7}, we apply cuboid K functional to get K functional for the case $p_0=p_1$.
As far as we know, not only has no one systematically established the wavelet characterization of the K functional in this case,
but also no one has worked out what the interpolation space is for the special cases $s_0=s_1$ and $r\neq q$.
For $p_0=p_1$, we can see that
the real interpolation spaces are not Besov spaces for $s_0=s_1$ and $r\neq q$
and the rest real interpolation spaces are known as Besov spaces.
In Section \ref{sec:6.2},
we solidify the wavelet coefficients defined on the layer grid $\Gamma_j$
and we transform the K functional into the K functional of the weighted space on the main grid $\mathbb{N}$.
In Sections \ref{sec:7.1}, we study the cases where $q_0=q_1$ and $s_0\neq s_1$.
Firstly, in Lemma \ref{lem:7.4}, we prove the vertex K functional theorem for weighted $(l^{a,q}, l^{b,q})$.
Then in Lemma \ref{lem:q}, we calculate the expression of the specific K functional for weighted $(l^{a,q}, l^{b,q})$.
And in Theorem \ref{lem:7.6}, we give the real interpolation space of weighted $(l^{a,q}, l^{b,q})$.
Finally in Corollary \ref{cor:7.7},
we give the real interpolation spaces of the corresponding Besov spaces and the K functional defined by the wavelet coefficients.
In Section \ref{sec:7.11}, we study the cases where $s_0\neq s_1$ and $q_0\neq q_1$.
In Theorem \ref{lem:layin}, we get the real interpolation spaces $(l^{s_0,q_0}, l^{s_1,q_1})_{\theta,r}$ by applying
Lemma \ref{lem:q} and Lions-Peetre Iteration Theorem.
In Theorem \ref{lem:q01}, we obtain the K functional $K(t,f, l^{s_0,q_0}, l^{s_1,q_1})$ by applying
Lemma \ref{lem:q} and Holmstedt Theorem.
In Corollary \ref{cor:7.10}, for $s_0\neq s_1$ and $q_0\neq q_1$, we get
the corresponding real interpolation space is still Besov space
and we obtain the relative K functional.
In Section \ref{sec:7.2}, we study the cases  $p_0=p_1$ and $s_0=s_1$.
With the help of cuboid functional,
we reveal that the relative real interpolation spaces for the cases  $p_0=p_1$ and $s_0=s_1$
are the spaces where the usual norms of the frequency are taken first,
and then the Lorentz index are taken for the obtained norms.

\vspace{0.2cm}
In Section \ref{sec:8}, we calculate the K functional for the cases where $p_0<p_1$ and $q_0=q_1$ based on cuboid K functional.
This is because, for the $K_q$ functional,
the functional on the layer grid and the summation over the main grid are commutative.
We get the relative K functional in
the equations \eqref{eq:982} and \eqref{eq:983}
defined by the summation of a group of K functional at different single layer
and in Theorem \ref{th:8.3}.
We generalized Lou-Yang-He-He's Theorem \cite{LYHH} to the case $s_0\neq s_1$.

\vspace{0.2cm}
In Section \ref{sec:9},
we use the quasi-trigonometric inequality satisfied by the norm of the function spaces
to deeply study the structure of the cuboid functional
and obtain the vertex K functional.
Therefore, we transform the functional problem into the classification problem of wavelet coefficients,
which provides support for K functional analysis by analyzing the structure of lattice topology.

\vspace{0.2cm}
In Section \ref{sec:9b}, we use vertex K functionals
to  give a short proof of the equivalent wavelet characterization of K functionals
when $q_0=q_1$.
Finally, in Section \ref{sec:10},
we consider vertex $K_{\infty}$ functional.
By introducing two kinds of power spaces,
and by this change of topology,
we get two sets of intermediate spaces $(X_0,X_1)$ and $(Y_0^j, Y_1^j)$.
See Lemmas \ref{cor:11.2} and \ref{cor:11.1}.
The functional of intermediate space has the topological compatibility.
See Lemma \ref{lem:ab}.
Finally, {\bf the wavelet characterization Theorem \ref{th:11.4} of all the K functional for $0<p_0\neq p_1\leq \infty$ and $0<q_0\neq q_1<\infty$
is given by using vertex functional and power spaces
which contain three levels of nonlinear topology meaning.}

\section{Preliminaries on interpolation spaces}\label{sec:2}
In this section, we present some preliminaries on interpolation spaces.
\subsection{K functional and basic properties}\label{sec:2.1}
We recall that  $(A_{0},\;A_{1})$ is a pair of quasi-normed spaces which are continuously embedded in a Hausdorff space $X$
and assume that $A_0$ and $A_1$ are compatible normed vector spaces. See \cite{BL, CDL, Peetre, TriebelB, YCP}.
For any threshold $t>0$,
$\forall f\in A_0 + A_1$ and $f_0\in A_0$,
define  a set of dynamically varying norms
$$ K(t,f, A_0, A_1,f_0) =  \|f_0\|_{A_0} + t \|f-f_0\|_{A_1} $$
for all $f\in A_{0}+A_{1}$ with $f_{0}\in A_{0}$ and $f_{1}\in A_{1}$.
$K$-functional is the lower bound of this dynamically varying norm, i.e.
$$K:= K(t,f, A_0, A_1) = \inf_{f=f_0+ f_1}
\{ \|f_0\|_{A_0} + t \|f_1\|_{A_1}\}.$$
$K(t,f,  A_0, A_1)$ is a nonnegative, increasing concave function of $t$.
K functionals are very similar to cost functions with regularized terms in learning algorithms, see \cite{LCWG}.

\begin{remark}
In learning algorithms, $ \|f_0\|_{A_0}$ is often taken as the empirical mean,
the norm of $A_1$ is chosen often as $L^{p}$ norm,
and the threshold $t$ is a relatively small quantity.
The unlabeled data is taken in $A_0+A_1$,
and the processed data is labeled $+1$ in $A_0$ and labeled $-1$ in $A_1$.
In this paper,
the nonlinear functional structure of K functional is analyzed and converted into vertex functional,
and it is concluded that the essence of K functional is data classification.
Then the data is classified by the lattice topology.
In learning algorithms, the way to achieve the goal of classification is through a lot of computation.
The learning algorithm hopes to have an effective and simple way to label data
\cite{HMY, HWJJ}.

{\bf Before our work}, people used {\bf gradient descent} to figure out how to minimize the cost function.
However, in the era of big data,
the amount of computation is too large, so {\bf distributed algorithm} is adopted.
Various efforts to optimize distributed algorithms appear.
However, these algorithms do not probe into the intrinsic mathematical structure of cost function,
so they cannot best match cost function.
This paper changes this approach.
Through some nonlinear functional analysis and nonlinear topological flow analysis,
the calculation of K-functional only needs to consider the lattice topology.
{\bf Sorting data by calculating the grid topology will greatly reduce the amount of computation}.
Therefore, the study of K functional is one of the theoretical bases of learning algorithm.
\end{remark}

K functionals have the following commutative properties:
\begin{lemma} \label{lem:CK}
For $t>0$, we have
$K(t,f, A_0, A_1)= t K(t^{-1}, f, A_1, A_0).$
\end{lemma}

Based on K functional, we can define the corresponding real interpolation space. For $0<\theta<1$,
$$\|f\|_{(A_0,A_1)_{\theta,\infty,K}}= \sup\limits_{t>0} t^{-\theta} K(t, f, A_0, A_1).$$
$$(A_0, A_1)_{\theta,\infty, K} = \{f: f\in A_0 +A_1, \|f\|_{(A_0,A_1)_{\theta,\infty, K}}<\infty\}.$$
Further, for $0<\eta<\infty$, define
$$\|f\|_{(A_0,A_1)_{\theta,\eta,K}}=\lf \{\int^{\infty}_{0} [ t^{-\theta} K(t, f, A_0, A_1)] ^{\eta}
\frac{dt}{t}\r \}^{\frac{1}{\eta}}.$$
$$(A_0, A_1)_{\theta,\eta, K} = \{f: f\in A_0 +A_1, \|f\|_{(A_0,A_1)_{\theta,\eta, K}}<\infty\}.$$
The interpolation spaces have the following commutative properties:
\begin{lemma}\label{lem:commutative}
Given $0<\theta<1$ and $ 0<\eta\leq \infty$. Then
\begin{equation*}\label{eq:2221}
\|f\|_{(A_0,A_1)_{\theta,\eta, K}}= \|f\|_{(A_1,A_0)_{1-\theta,\eta, K}}.
\end{equation*}
\end{lemma}

The following Lions-Peetre's iterative Theorem  has been used often to get complex real interpolation spaces from simple real interpolation spaces.
\begin{lpthm}\label{lem:LP}
For $0<\theta_0, \theta_1, \eta<1, 0<q_0, q_1\leq \infty, 1\leq q\leq \infty$  and\;　$\theta= (1-\eta) \theta_0 + \eta \theta_1$, we have
$$\{ (A_0, A_1) _{\theta_0, q_0}, (A_0, A_1)_{\theta_1, q_1}\}_{\eta, q} = (A_0, A_1)_{\theta, q}.$$
\end{lpthm}
See \cite{BL, LionsP, Peetre}.
Furthermore, we can use
the following Lemma proved by Holmstedt in \cite{Holmstedt}
to compute the relative K functional.
\begin{holthm} \label{th:Holmstedt}
Let $A_0$ and $A_1$ be a couple of quasi-normed spaces and for $i=1,2$, write
$E_{i}= (A_0, A_1)_{\theta_i, q_i} $.
If $0<\theta_0<\theta_1<1$, $\eta= \theta_1-\theta_0$ and $0<q_0,q_1\leq \infty$, then
$$\begin{array}{c}
K(t,f, E_0, E_1)\\
=  \lf\{ \int^{t^{\frac{1}{\eta}}}_{0} (s^{-\theta_0} K(s, f, A_0, A_1))^{q_0} \frac{ds}{s}\r\}^{\frac{1}{q_0}}
+ t \lf\{\int^{\infty}_{t^{\frac{1}{\eta}}} (s^{-\theta_1} K(s, f, A_0, A_1))^{q_1} \frac{ds}{s}\r\}^{\frac{1}{q_1}}.
\end{array} $$
\end{holthm}

The calculation of many K functionals is quite complex.
Holmstedt-Peetre \cite{HP} introduced $K_{\xi}$ functional to improve the situation.
For $0<\xi<\infty$, define the functional $K_{\xi}$ by
$$K_{\xi}=K_{\xi}(t,f,A_{0},A_{1})=\inf\limits_{f=f_{0}+f_{1}}(\|f_{0}\|_{A_{0}}^{\xi}+t^{\xi}\|f_{1}\|_{A_{1}}^{\xi})^{\frac{1}{\xi}}.$$
For $\xi=\infty$, define
$$K_{\infty}=K_{\infty}(t,f,A_{0},A_{1})=\inf\limits_{f=f_{0}+f_{1}}\max (\|f_{0}\|_{A_{0}}, t\|f_{1}\|_{A_{1}}).$$
And for $0<\xi\leq \infty$,
$$
\|f\|_{(A_{0},A_{1})_{\theta,q,K_{\xi}}}=\left\{\int_{0}^{\infty}[t^{-\theta}K_{\xi}(t,f,A_{0},A_{1})]^{q}\frac{dt}{t}\right\}^{\frac{1}{q}}.
$$
Holmstedt-Peetre \cite{HP} tell us that the real interpolation spaces formed by
the two kinds of K functionals are equivalent.
\begin{lemma}\label{le3.1}
Let $(A_{0},A_{1})$ be a couple of quasi-normed spaces.
For any $0<\xi \leq \infty$, we have
$$\|f\|_{(A_{0},A_{1})_{\theta,q,K}}\sim\|f\|_{(A_{0},A_{1})_{\theta,q,K_{\xi}}}.$$
\end{lemma}

\subsection{K functional for Lebesgue functions}\label{sec:2.2}
Lorentz \cite{Lorentz} has considered the distribution function and the non-increasing rearrangement function of Lebesgue functions
and introduced Lorentz spaces.
For set $E$, let $|E|$ denote the measure of $E$.
Consider the following right continuons non-increasing functions:\\
$$\sigma_{f}(\lambda)= |\{x: |f(x)|>\lambda\}| {\rm\, and \,} f^{*}(\tau)= \inf \{\lambda: \sigma_{f}(\lambda)\leq \tau\}.$$
For distribution function and rearrangement function, there is the following relation,
see page 231 of Cheng-Deng-Long's book \cite{CDL}.
\begin{lemma} \label{th:CDL}
Let $\Phi(u)$ be a continuously increasing function satisfying $\Phi(0)=0$.
For any measurable function $f$ and $\lambda>0$ on the general measure space $(X, \mathfrak{F}, \mu)$,
\begin{equation}\label{eq:(8)}
\int_{\{x: |f|\leq \lambda\}} \Phi(|f|) d\mu= \int^{\infty}_{\sigma_{f}(\lambda)} \Phi(f^{*}(t)) dt.
\end{equation}
\begin{equation}\label{eq:(9)}
\int_{\{x:|f|> \lambda\}} \Phi(|f|) d\mu= \int^{\sigma_{f} (\lambda)}_{0} \Phi(f^{*}(t)) dt.
\end{equation}
\end{lemma}

Lorentz space $L^{p,q}$ is defined by non-increasing function:
\begin{definition}\label{def:2.7}
Given $0<p,q<\infty$. Then\\
{\rm (i)} $f\in L^{p,q} \Leftrightarrow \{ \frac{q}{p} \int^{\infty}_{0}  (t^{\frac{1}{p}} f^{*}(t))^{q} \frac{dt}{t} \}^{\frac{1}{q}}< +\infty.$\\
{\rm (ii)} $f\in L^{p,\infty} \Leftrightarrow \sup\limits_{t>0} t^{\frac{1}{p}} f^{*}(t)< +\infty.$
\end{definition}
Hunt and Holmstedt have found out the following K functional of Lebesgue function
by using the rearrangement functions. See \cite{BL, Lorentz, Peetre}.
\begin{lemma}\label{lem:2.5}
{\rm (i)} If $0<p<\infty$, then $K(t,f, L^{p}, L^{\infty}) = [ \int ^{t^{p}}_{0} | f^{*}(\tau)| ^{p} d\tau ]^{\frac{1}{p}}.$\\
{\rm (ii)} If $0<p_{0}<p_{1}<\infty$ and $\frac{1}{\alpha}= \frac{1}{p_{0}}- \frac{1}{p_{1}}$, then\\
$$K(t,f, L^{p_{0}}, L^{p_{1}}) = \lf[ \int ^{t^{\alpha}}_{0} | f^{*}(\tau)| ^{p_{0}} d\tau\r ]^{\frac{1}{p_{0}}}
+ t\lf [ \int ^{\infty}_{t^{\alpha }} | f^{*}(\tau)| ^{p_{1}} d\tau\r ]^{\frac{1}{p_{1}}}.$$
\end{lemma}

We will later develop the discrete form of this lemma without proof.
By Lemma \ref{th:CDL},
there exist following relations between the big data set of the function and the rearrangement function.
\begin{lemma}\label{lem:2.6}
For $\lambda>0$, we have\\
{\rm (i)} $ f^{*}(0)= \sup |f|$.\\
{\rm (ii)} $\int_{\{x: |f|\leq \lambda\}} |f|^{p} dx =
\int ^{+\infty} _{\sigma(\lambda)} |f^{*}(t)|^{p} dt.$\\
{\rm (iii)} $ \int_{\{x: |f|> \lambda\}} |f|^{p} dx =
\int ^{\sigma(\lambda)}_{0} |f^{*}(t)|^{p} dt.$
\end{lemma}

Combining the above two lemmas,
we know that the K functional is determined by the norm of the large value point part and the small value point part
bounded by a certain value in the corresponding space.
{\bf Later we will use the discrete form of these lemmas to classify the vertices of cuboids}.
The real interpolation spaces of all Lebesgue spaces have been known.
According to the above lemmas \ref{lem:2.5} and \ref{lem:2.6},
the real interpolation spaces of Lebesgue spaces are Lorentz spaces in Definition \ref{def:2.7}
and the following Lemma extend Marcinkiwicz theorem to Lorentz spaces.
See \cite{BL, CDL, Hunt1,Hunt2, Triebel, TriebelB}.
\begin{hthm}\label{lem:LL}
If $0<\theta<1, 0<q\leq \infty, 0<p_{0}<p_{1}\leq \infty$ and $\frac{1}{p}= \frac{1-\theta}{p_0} + \frac{\theta}{p_1},$ then
$$( L^{p_{0}}, L^{p_{1}})_{\theta,q}  = L^{p,q}.$$
\end{hthm}

The real interpolation spaces of Lebesgue spaces have been well studied.
But for more general function spaces, for example,
Besov spaces, Triebel-Lizorkin spaces,
Besov-Lorentz spaces and Tribel-Lizorkin-Lorentz spaces,
the non-linear structure relative to interpolation is very complex and
the corresponding interpolation spaces are far from clear.
See \cite{BL, Peetre, Triebel, TriebelB}.
Yang-Cheng-Peng \cite{YCP} use wavelets, vector valued Fefferman-Stein maximum operator and Whitney decomposition
to consider some properties of Tribel-Lizorkin-Lorentz spaces and Besov-Lorentz spaces.
In this paper, we consider systematically the interpolation spaces for Besov spaces by wavelets,
characterize all the nonlinear structure of the real interpolation of Besov spaces
and solve Peetre's conjecture.
Further the Lemma \ref{lem:dtod} tells us that our ideas can be applied to more general distribution spaces.

\section{Preliminaries on Peetre's conjecture} \label{sec:3}
In this section, we first recall the definition of homogeneous and non-homogeneous Besov type spaces
and then present what is Peetre's conjecture on the real interpolation spaces of Besov spaces.

\subsection{Besov type space}\label{sec:3.1}
Homogeneous and non-homogeneous Besov type spaces are based on distribution theory.
Let $S$ be the space of all Schwartz functions on $\mathbb{R}^{n}$.
Let $S_{0}$ be the space of all Schwartz functions $f$ on $\mathbb{R}^{n}$
such that $\int x^{\alpha} f(x) dx=0, \forall \alpha\in \mathbb{N}^{n}$.
The space of all tempered distributions on $S(\mathbb{R}^n)$
which is equipped with the weak-$\ast$ topology is denoted by $\mathcal{S}'(\mathbb{R}^n)$.
Respectively, the space of all tempered distributions on $S_{0}(\mathbb{R}^{n})$
which is equipped with the weak-$\ast$ topology is denoted by
$\mathcal{S}_{0}'(\mathbb{R}^n)$.
Hence $\mathcal{S}'(\mathbb{R}^{n})\subset \mathcal{S}_{0}'(\mathbb{R}^n).$
Denote the space of all polynomials on $\mathbb{R}^{n}$ by $P(\mathbb{R}^{n})$.
Then we have $\mathcal{S}_{0}'(\mathbb{R}^n)=\mathcal{S}'(\mathbb{R}^{n})\backslash P(\mathbb{R}^{n})$.

For Besov type spaces, we use the definition based on the Littlewood-Paley decomposition, see \cite{TriebelB}.
Given a nonnegative function $\widehat{\varphi}(\xi)\in S(\mathbb{R}^{n})$ such that
${\rm supp}\;\widehat{\varphi}\subset\{\xi\in \mathbb{R}^{n}:|\xi|\leq 2\}$ and $\hat{\varphi}(\xi)=1$ if $|\xi|\leq\frac{1}{2}$.
Define
$$\varphi_{u}(x)=2^{n(u+1)}\varphi(2^{u+1}x)-2^{nu}\varphi(2^{u}x).$$
Then the family of functions $\hat{\varphi}(\xi), \{\hat{\varphi}_{u}(\xi)\}_{u\in \mathbb{Z}}$ satisfy
\begin{equation*}
\left\{ \begin{aligned}
&{\rm supp}\;\hat{\varphi}_{u}\subset\{\xi\in \mathbb{R}^{n},\;\frac{1}{2}\leq 2^{-u}|\xi|\leq2\}.& \\
&|\hat{\varphi}_{u}(\xi)|\geq C>0, \;{\rm if} \;\frac{1}{2}<C_{1}\leq 2^{-u}|\xi|\leq C_{2}<2.&\\
&|\partial^{k}\hat{\varphi}_{u}(\xi)|\leq C_{k}2^{-u|k|}, \;{\rm for}\; {\rm any}\;k\in \mathbb{N}^{n}.&\\
&\sum\limits_{u=-\infty}^{+\infty}\hat{\varphi}_{u}(\xi)=1,\;{\rm for}\;{\rm all} \;\xi\in \mathbb{R}^{n}\setminus\{0\}.&\\
&\hat{\varphi}(\xi)+\sum\limits_{u\geq 0}\hat{\varphi}_{u}(\xi)=1,\;{\rm for}\;{\rm all} \;\xi\in \mathbb{R}^{n}.&
\end{aligned}
\right.
\end{equation*}

If $f\in S_{0}'(\mathbb{R}^{n})$,  for all $u\in \mathbb{Z}$, define $ f^{u}=\varphi_{u}\ast f$.
If $f\in S'(\mathbb{R}^{n})$, for any $u\geq 1$, define $ f_{u}=\varphi_{u}\ast f$
and  $f_{0}= (\varphi+ \varphi_0)\ast f. $
 $f^{u}\ and\ f_{u}$ are called the $u-$th dyadic block of the Littlewood-Paley decomposition of $f$.

For $s\in \mathbb{R}$ and $0<p\leq \infty$, we say $f$ belongs to the Sobolev spaces $\dot{W}^{s,p}$,
if $(-\Delta)^{-\frac{s}{2}}f \in L^{p}$.
We recall  the definition of $\dot{B}^{s,q}_{p}$ and $B^{s,q}_{p}$.
Besov spaces are defined by the weighted $l^{s,q}$ norm of the sequence of Lebesgue norms $L^{p}$ over a ring,
or equivalents, by the sequence $l^{q}$ norm of the Sobolev norm $W^{s,p}$ over a ring.
 See \cite{CDL, Triebel, TriebelB,Yang1}.
That is to say,

\begin{definition}\label{de1.1}
Let  $0< p, q\leq\infty$ and $s\in \mathbb{R}$. We have\\
{\rm(\romannumeral1)}
$f(x)\in \dot{B}^{s,q}_{p}$,   if $f\in S_{0}'(\mathbb{R}^{n})$ and   $\left[\sum\limits_{u\in \mathbb{Z}}2^{us q}\|f^{u}(x)\|_{L^{p}}^{q}\right]^{\frac{1}{q}}<\infty$.\\
{\rm(\romannumeral2)} $f(x)\in B^{s,q}_{p}$,  if $f\in S'(\mathbb{R}^{n})$ and  $\left[\sum\limits_{u\geq 0}2^{us q}\|f_{u}(x)\|_{L^{p}}^{q}\right]^{\frac{1}{q}}<\infty$.\\
When $q=\infty$,  it should be replaced by the supremum.
\end{definition}

We recall the definition of  Besov-Lorentz spaces which have been studied in \cite{BL, Peetre, YCP}.
Let $E\subset \mathbb{R}^{n}$, we denote $|E|$ the Lebesgue measure of $E$.
\begin{definition}\label{de1.2}
Assume that $0< p, q,r\leq\infty$, $s\in \mathbb{R}$ and $u,v\in \mathbb{Z}$. Then\\
(\romannumeral1)  $f(x)\in \dot{B}^{s,q}_{p,r}$, if $f\in S_{0}'(\mathbb{R}^{n})$ and $$\left[\sum\limits_{u\in \mathbb{Z}}2^{uqs}\left(\sum\limits_{v\in \mathbb{Z}}2^{rv}|\{x\in \mathbb{R}^{n}:|f^{u}(x)|>2^{v}\}|^{\frac{r}{p}}\right)^{\frac{q}{r}}\right]^{\frac{1}{q}}<\infty.$$
(\romannumeral2)  $f(x)\in B^{s,q}_{p,r}$, if $f\in S'(\mathbb{R}^{n})$ and  $$\left[\sum\limits_{u\geq 0}2^{uqs}\left(\sum\limits_{v\in \mathbb{Z}}2^{rv}|\{x\in \mathbb{R}^{n}:|f_{u}(x)|>2^{v}\}|^{\frac{r}{p}}\right)^{\frac{q}{r}}\right]^{\frac{1}{q}}<\infty.$$
As $p=\infty$ or $q=\infty$ or $r=\infty$,  it should be modified by  supremum.
\end{definition}

The definition of the above spaces are independent of the choice of  $\varphi_{u}$,
see \cite{BL,TriebelB,Yang1,Yang2,YCP}.

\subsection{
Peetre's conjecture}\label{sec:3.2}
Peetre \cite{Peetre} applied Lions-Peetre's iterative Theorem of interpolation and other skills
to get a part of real interpolation spaces of Besov spaces.
\begin{theorem}\label{lem:3.3}
{\bf Theorem 6 at page 106 of  Peetre \cite{Peetre}. }
Let $0<q_0, q_1\leq \infty, 0<\theta<1$, $\frac{1}{q}= \frac{1-\theta}{q_0}+ \frac{\theta}{q_1}$ and $0<r\leq \infty.$\\

{\rm (i)} Given $1\leq p\leq \infty$ and $s_0\neq  s_1\in \mathbb{R}$, then
$$(B^{s_0,q_0}_{p}, B^{s_1,q_1}_{p})_{\theta, r} = B^{s,r}_{p}, {\rm \; if \; } s= (1-\theta) s_0 + \theta s_1.$$

{\rm (ii)} Given $1\leq p\leq \infty$ and $s\in \mathbb{R}$, then
$$(B^{s,q_0}_{p}, B^{s,q_1}_{p})_{\theta, r} = B^{s,r}_{p}, {\rm \; if \; } r=q.$$

{\rm (iii)} Given $1\leq p_{0}\neq  p_{1} \leq \infty,  s_0, s_1\in \mathbb{R}$ and
$\frac{1}{p}= \frac{1-\theta}{p_0}+ \frac{\theta}{p_1}$, then
$$(B^{s_0,q_0}_{p_0}, B^{s_1,q_1}_{p_1})_{\theta, r} = B^{s,q}_{p,q}, {\rm \; if \; } r=q.$$
\end{theorem}

Note the particularity of $q$, the condition $r=q$ in Theorem \ref{lem:3.3}
implies that very few indices are known for the real interpolation spaces.
One did not know how to deal with the case where $r\neq q$.
\begin{conjecture} \label{con:1}
Given $1\leq p_{0}, p_{1} \leq \infty$,
$0<q_0, q_1\leq \infty$, $\frac{1}{q}= \frac{1-\theta}{q_0}+ \frac{\theta}{q_1}, 0<\theta<1$ and $0<r\leq \infty.$
Peetre proposed to find out the expression of
$(B^{s_0,q_0}_{p_0}, B^{s_1,q_1}_{p_1})_{\theta, r}$ for $r\neq q$
at page 110 in his book \cite{Peetre} of 1976.
\end{conjecture}

The resolution of this conjecture faces many difficulties.
The norm of function space only reflects the continuous structure of functions,
but K functional reflects the internal geometric structure of functions and non-linearity.
In the process of dealing with this longstanding open problem,
we found that
{\bf K-functional reflects two types of nonlinearity at the same time:
functional nonlinearity and topological nonlinearity.}

\section {Preliminaries on wavelets}\label{sec:4}
In this section, we introduce the preliminaries on wavelets.
We recall the wavelet characterization for Besov type spaces
and progress on real interpolation spaces of Besov spaces.

\subsection{Wavelets and Besov type spaces}\label{sec:4.1}
We use regular orthogonal tensorial wavelets. See \cite{Meyer, Yang1, YCP}.
For dimension 1, let $\Phi^{0}$ be the father wavelet and let $\Phi^{1}$ be the mother wavelet.
For dimension $n$, let directional set $\Xi= \{0,1\}^{n}$. For $ \epsilon\in \Xi$, denote $\Phi^{\epsilon}(x) = \prod\limits^{n}_{i=1} \Phi^{\epsilon_{i}}(x_{i})$.
For $j\in \mathbb{Z}, \epsilon\in \Xi, k\in \mathbb{Z}^{n}$, let
$\Phi^{\epsilon}_{j,k}(x) = 2^{\frac{nj}{2}} \Phi^{\epsilon}( 2^{j} x-k).$
$\epsilon$ denotes the property of whether the integral of the direction function is zero for each coordinate axis,
$j$ is a quantity related to frequency, $k$ is a quantity related to position.
To characterize both homogeneous spaces and non-homogeneous spaces, we need the following notations:
$$\dot{\Xi}= \{ \epsilon: \epsilon\in \{0,1\}^{n} \backslash \{(0,\cdots,0)\}\},
\Xi_{0}= \Xi  {\rm \, and \, for \,} j\geq 1, \Xi_{j}=\dot{\Xi}$$
$$\dot{\Gamma}= \{ \gamma= (\epsilon, k), \epsilon\in \dot{\Xi}, k\in \mathbb{Z}^{n}\},
\mathbb{N}= \{j\in \mathbb{Z}, j\geq 0\}$$
$$\Gamma_{0}=\Gamma = \{ \gamma= (\epsilon, k), \epsilon\in \Xi_{0}, k\in \mathbb{Z}^{n}\},
{\rm and \, for \,} j\geq 1, \Gamma_{j} = \dot{\Gamma}.$$
$$\dot{\Lambda}= \{(j,\gamma), j\in \mathbb{Z}, \gamma\in \dot{\Gamma}\} {\rm \, and \,}
\Lambda = \{(j,\gamma), j\in \mathbb{N}, \gamma\in \Gamma_{j}\}.$$
We do not distinguish between $(\epsilon,j,k)$ and $(j,\gamma)$ and denote
$\Phi_{j,\gamma}(x) = \Phi^{\epsilon}_{j,k}(x)$. Denote
$f_{j,\gamma}= \langle f, \Phi_{j,\gamma}\rangle $.

We use indices set $\dot{\Lambda}$ to consider homogeneous spaces
and  indices set $\Lambda$ to consider non-homogeneous spaces.
For $\Lambda$, we specifically point out that $j$ represents the frequency layer lattice of the object being studied,
and denote $\Gamma_{j}$ the marked grid which represents the direction and position lattice.
In particular, we point out that wavelet has two characteristics: (1) wavelet basis $\Phi_{j,\gamma}(x)$ and (2) discrete grid
$\Lambda$ which is composed by main grid $\mathbb{Z}$ and layer grid $\Gamma_j$.
The former is used for nonlinear functional research and the latter for lattice topology research.

By using wavelets knowledge in \cite{Meyer, Yang1},
 Lemma \ref{lem:dtod}  provides the basis for studying the discretization of the general distribution space in below.
\begin{lemma}\label{lem:dtod}
If we use Meyer wavelets, then wavelet theory shows :\\
{\rm (i)} $f\rightarrow \{f_{j,\gamma}\}_{(j,\gamma)\in \Lambda}$ is an isomorphic mapping from $S'(\mathbb{R}^{n})$ to $\mathbb{C}^{\Lambda}$.
\\
{\rm (ii)} $f\rightarrow \{f_{j,\gamma}\}_{(j,\gamma)\in \dot{\Lambda}}$ is an isomorphic mapping from $S'_{0}(\mathbb{R}^{n})$ to $\mathbb{C}^{\dot{\Lambda}}$.
\end{lemma}
So without causing confusion, we do not distinguish between the function $f$ and the sequence $f_{j,\gamma}$.
Furthermore, we have the following characterization for Besov spaces:
\begin{lemma} \label{lem:cbesov}
Assume that $0< p, q \leq\infty$ and $s\in \mathbb{R}$. Then
$$f(x)= \sum\limits_{(j,\gamma)\in \dot{\Lambda}} f_{j,\gamma} \Phi_{j,\gamma}\in \dot{B}^{s,q}_{p} \Leftrightarrow
\sum\limits_{j\in \mathbb{Z}} 2^{jq(s+\frac{n}{2}-\frac{n}{p})} (\sum\limits_{\gamma\in \dot{\Gamma}}
|f_{j,\gamma}|^{p})^{\frac{q}{p}} < +\infty.$$
$$f(x)= \sum\limits_{(j,\gamma)\in \Lambda} f_{j,\gamma} \Phi_{j,\gamma}\in B^{s,q}_{p} \Leftrightarrow
\sum\limits_{j\geq 0} 2^{jq(s+\frac{n}{2}-\frac{n}{p})} (\sum\limits_{\gamma\in \Gamma_{j}}
|f_{j,\gamma}|^{p})^{\frac{q}{p}} < +\infty.$$
\end{lemma}

Based on  Lemmas \ref{lem:dtod} and \ref{lem:cbesov},
we specifically define the following sequence spaces which will be used later.
\begin{definition}\label{def:discretespace}
Assume that $0< p, q \leq\infty$ and $s\in \mathbb{R}$. Then

{\rm (i)} $\{f_{j,\gamma}\} \in l^{s,q}_{p}= l^{s,q}_{p}(\Lambda) \Leftrightarrow \sum\limits_{j\geq 0} 2^{jqs} (\sum\limits_{\gamma\in \Gamma_{j}} |f_{j,\gamma}|^{p})^{\frac{q}{p}} < +\infty.$

{\rm (ii)} $\{f_{j,\gamma}\} \in \tilde{l}^{s,q}_{p}= \tilde{l}^{s,q}_{p}(\Lambda) \Leftrightarrow \{f_{j,\gamma}\} \in l^{s+\frac{n}{2}-\frac{n}{p},q}_{p}= l^{s+\frac{n}{2}-\frac{n}{p},q}_{p}(\Lambda) .$

{\rm (iii)} $\{f_{j}\} \in l^{s,q}=l^{s,q}(\mathbb{N}) \Leftrightarrow \{\sum\limits_{j\in \mathbb{N}} 2^{jsq} |f_{j}|^{q}\}^{\frac{1}{q}}<\infty.$

\end{definition}
 Lemma \ref{lem:cbesov} establishes a one-to-one correspondence between functions in Besov spaces on $\mathbb{R}^{n}$
and discrete sequences on grid $\dot{\Lambda}$ or $\Lambda$.
Lemma \ref{lem:cbesov} implies the following result.
\begin{corollary}
Assume that $0< p, q \leq\infty$ and $s\in \mathbb{R}$. Then
$$f(x) \in B^{s,q}_{p}\Leftrightarrow \{ f_{j,\gamma} \}\in \tilde{l}^{s,q}_{p}= l^{s+\frac{n}{2}-\frac{n}{p},q}_{p}.$$
\end{corollary}

Furthermore, we recall the wavelet characterization of Besov-Lorentz spaces.
For $j\in \mathbb{Z}$, denote $f^{j}(x)= 2^{\frac{nj}{2}} \sum\limits_{\gamma\in \dot{\Gamma} } |f_{j,\gamma}| \chi (2^{j}x-k).$
For $j\in \mathbb{N}$, denote $f_{j}(x)= 2^{\frac{nj}{2}} \sum\limits_{\gamma\in \Gamma_{j} } |f_{j,\gamma}| \chi (2^{j}x-k)$
and  $f^{\infty}_{j}(x)= 2^{\frac{nj}{2}} \sup\limits_{\epsilon\in \Xi_{j}} \sum\limits_{k\in \mathbb{Z}^{n} } |f_{j,\gamma}| \chi (2^{j}x-k).$
According to \cite{YCP}, we have the following characterization for Besov-Lorentz spaces:
\begin{lemma}\label{lem:Besov-Lorentz}
Assume that $0< p, q,r\leq\infty$ and $s\in \mathbb{R}$. Then
$$f(x)= \sum\limits_{(j,\gamma)\in \dot{\Lambda}} f_{j,\gamma} \Phi_{j,\gamma}\in \dot{B}^{s,q}_{p,r} \Leftrightarrow
\sum\limits_{j\in \mathbb{Z}} 2^{jsq} \lf\{ \sum\limits_{u\in \mathbb{Z}} 2^{ur}
|\{x: f^{j}(x)> 2^{u}\}|^{\frac{r}{p}}\r\}^{\frac{q}{r}} < +\infty.$$
$$f(x)= \sum\limits_{(j,\gamma)\in \Lambda} f_{j,\gamma} \Phi_{j,\gamma}\in B^{s,q}_{p,r} \Leftrightarrow
\sum\limits_{j\in \mathbb{N}} 2^{jsq}\lf \{ \sum\limits_{u\in \mathbb{Z}} 2^{ur}
|\{x: f_{j}(x)> 2^{u}\}|^{\frac{r}{p}}\r\}^{\frac{q}{r}} < +\infty. $$
\end{lemma}

\subsection{Known advances on the Peetre conjecture}\label{sec:4.2}
In 1974, Cwikel \cite{C} proved that
the Lions-Peetre formula have no reasonable generalization for any $r\neq q$.
Peetre conjecture never got off the ground because of the lack of a way to deal with its complex nonlinearity.
Until recently,
Lou-Yang-He-He \cite{LYHH} considered the particular indices $s_0=s_1=s\in \mathbb{R}, 0<q_0=q_1=q\leq \infty$ and $p_0\neq p_1$
by exploiting the commutativity of the $K_{q}$ functional with respect to frequency level and the infimum.
Let $\chi(x)$ be the unit cube $[0,1]^{n}$.
Let $$c_{j,n}(\tau)= \inf \lf[ \lambda: |\{ x\in \mathbb{R}^{n}, 2^{\frac{nj}{2}}
\sum\limits_{\gamma\in \dot{\Gamma}} |f_{j,\gamma}| \chi(2^{j}x-k)> \lambda \} | \leq \tau\r].$$
For $1<p_0, p_1<\infty$ and $\frac{1}{\alpha}= \frac{1}{p_0}-\frac{1}{p_1}$, denote
$$b^{p_0,p_1}_{j,n,u} = \lf[\int^{ 2^{u\alpha} }_{0} |c_{j,n}(\tau)|^{p_0} d\tau \r]^{\frac{1}{p_0}}
+\lf[\int^{\infty}_{2^{u\alpha}} |c_{j,n}(\tau)|^{p_1} d\tau \r]^{\frac{1}{p_1}} .$$
Based on wavelets and commutativity of summation and infimum,
Lou-Yang-He-He \cite{LYHH} obtained the following real interpolation spaces of homogeneous Besov spaces.
\begin{theorem}
Given $s\in \mathbb{R}, \theta\in (0,1)$, $0<q, r\leq \infty$, $1<p_0<p_1<\infty$ and
$\frac{1}{p}= \frac{1-\theta}{p_0}+ \frac{\theta}{p_1}$.
We have

{\rm (i)} $f(x)\in (\dot{B}^{s,q}_{p_0}, \dot{B}^{s,q}_{\infty})_{\theta, r}$ if and only if
$$\sum\limits_{u} 2^{-u r \theta} \lf\{ \sum\limits_{j\in \mathbb{Z}} 2^{jsq}
\lf[\int^{2^{u p_0}}_{0} (c_{j,n}(\tau)) ^{p_0} d\tau \r]^{\frac{q}{p_0}} \r\}^{\frac{r}{q}}<\infty.$$

{\rm (ii)} $f(x)\in (\dot{B}^{s,q}_{p_0}, \dot{B}^{s,q}_{p_1})_{\theta, r}$ if and only if
$$\sum\limits_{u} 2^{-u\eta \theta} \lf\{ \sum\limits_{j\in \mathbb{Z}} 2^{jsq}
[b^{p_0,p_1}_{j,n,u} ]^{q} \r\}^{\frac{r}{q}}<\infty.$$

By Lemma \ref{lem:Besov-Lorentz}, when $r=q$, the above spaces are Besov-Lorentz spaces $\dot{B}^{s,q}_{p,q}.$

\end{theorem}

In addition, Besoy-Haroske-Triebel \cite{BHT}
are also concerned about this longstanding open problem
and considered $(\dot{B}_{p_0}^{s_0 , q_{0}}, \dot{B}_{p_1}^{s_1, q_{1}})_{\theta, r}$
where $p_0=q_0$, $p_1=q_1$  and $s_0-s_1= \frac{\alpha}{p_0}-\frac{\alpha}{p_1}$.
In fact, they used wavelets and trace theorem based on the results in \cite{BCT, YCP} to get
$(\dot{B}_{p_0}^{s-\alpha / p_0, p_{0}}, \dot{B}_{p_1}^{s- \alpha / p_1, p_{1}})_{\theta, r}$
where $0<p_0<p_1<\infty, 0<\theta<1, -\infty<\alpha<\infty, 0<r\leq \infty$ and $s\in \mathbb{R}$.

\section{Cuboid K functional} \label{sec:6}

In this section,
we establish the wavelet functional
and obtain the cuboid functional
by analyzing the coefficient characteristics of the approximation wavelet functional.
It is concluded that wavelet is an unconditional basis for interpolation space.
Devore-Popov have considered the relation of K functional
between Besov spaces and the corresponding discrete spaces for frequency in Theorem 6.1 of \cite{DP}.
The following Theorem \ref{th:6.1} gives a fully discrete form of their result.
Then through a series of analysis,
the object of obtaining functional is gradually reduced,
so that the functional is limited to the cuboid,
and the cuboid functional is obtained.

Continuous K functional is the lower bound of the dynamic norm from $(\mathcal{S}', \mathcal{S}')$ to $\mathbb{R}_{+}$.
By using wavelets, we transform the dynamic norm to the discrete dynamic norm
from $(\mathbb{C}^{\Lambda}, \mathbb{C}^{\Lambda})$ to $\mathbb{R}_{+}$.
Thus, we transform the relative K functional to the extreme value of sequence defined on grid $\Lambda$
which takes values in $\mathbb{C}^{\Lambda}$ and get discrete K functional.
In fact, by wavelet characterization, there exists a one to one correspondence between function and sequence.
Given $0<\xi<\infty$.
Because of Lemma \ref{le3.1}, we consider $K_{\xi}$ functional.
In fact, by wavelet characterization Lemma \ref{lem:cbesov} on Besov spaces,
we obtain the following equivalence between the K-functional of the real interpolation of Besov space and the K-functional of the real interpolation of the discretized space.
\begin{theorem}\label{th:6.1}
Given $s_0,s_1\in \mathbb{R}, 0<p_0,p_1,q_0,q_1\leq \infty$ and $0<\xi<\infty$.
K functional for $f \in B^{s_0,q_0}_{p_0}+ B^{s_1,q_1}_{p_1}\in S'(\mathbb{R}^{n})$
can be transformed to the infimum of sequence $\{f_{j,\gamma}\}$
defined on $\tilde{l}^{s_0,q_0}_{p_0}+ \tilde{l}^{s_1,q_1}_{p_1}\in \mathbb{C}^{\Lambda}$:
$$\begin{array}{l}
K_{\xi} (t,f, B^{s_0,q_0}_{p_0}, B^{s_1,q_1}_{p_1})=
\inf \limits_{g\in B^{s_0,q_0}_{p_0}} K_{\xi} (t,f, B^{s_0,q_0}_{p_0}, B^{s_1,q_1}_{p_1}, g)\\
= \inf \limits_{\{g_{j,\gamma}\}_{(j,\gamma)\in \Lambda}\in \tilde{l}^{s_0,q_0}_{p_0}} K_{\xi} (t,f, \tilde{l}^{s_0,q_0}_{p_0}, \tilde{l}^{s_1,q_1}_{p_1}, g_{j,\gamma})
=K_{\xi} (t,f, \tilde{l}^{s_0,q_0}_{p_0}, \tilde{l}^{s_1,q_1}_{p_1}).
\end{array}$$
\end{theorem}
\begin{proof}
For $f= (f_{j,\gamma})_{(j,\gamma)\in \Lambda}$ and $ g= (g_{j,\gamma})_{(j,\gamma)\in \Lambda}$,
consider the norm with perturbation under threshold $t$ defined by the sum of sequence
\begin{equation}\label{eq:Kxi}
\begin{array}{cl} & K_{\xi} (t,f, B^{s_0,q_0}_{p_0}, B^{s_1,q_1}_{p_1}, g)\\
=&
\lf\{\lf[\sum\limits_{j\geq 0} 2^{j q_0 (s_0+ \frac{n}{2}-\frac{n}{p_0})}
(\sum\limits_{\gamma\in \Gamma_j} g_{j,\gamma}^{p_0}) ^{\frac{q_0}{p_0}} \r] ^{\frac{\xi}{q_0}}\r.\\
& \lf.+ t^{\xi} \lf[\sum\limits_{j\geq 0} 2^{j q_1 (s_1+ \frac{n}{2}-\frac{n}{p_1})}
(\sum\limits_{\gamma\in \Gamma_j} (g_{j,\gamma}-f_{j,\gamma}) ^{p_1}) ^{\frac{q_1}{p_1}} \r] ^{\frac{\xi}{q_1}}\r\}^{\frac{1}{\xi}}\\
\equiv &
K_{\xi} (t,\{f_{j, \gamma}\}, \tilde{l}^{s_0,q_0}_{p_0}, \tilde{l}^{s_1,q_1}_{p_1}, \{g_{j,\gamma}\}).
\end{array}
\end{equation}
$K^{\xi}_{\xi} (t,f, B^{s_0,q_0}_{p_0}, B^{s_1,q_1}_{p_1}, g)$ is defined for
$f, g\in B^{s_0,q_0}_{p_0} + B^{s_1,q_1}_{p_1}\in S'(\mathbb{R}^{n})$.
But $K^{\xi}_{\xi} (t,\{f_{j,\gamma}\}, \linebreak \tilde{l}^{s_0,q_0}_{p_0}, \tilde{l}^{s_1,q_1}_{p_1}, \{g_{j,\gamma}\})$ is defined for
$\{f_{j,\gamma}\}_{(j,\gamma)\in \Lambda}, \{g_{j,\gamma}\}_{(j,\gamma)\in \Lambda} \in \tilde{l}^{s_0,q_0}_{p_0} + \tilde{l}^{s_1,q_1}_{p_1}\in \mathbb{C}^{\Lambda}.$
Thus $K_{\xi} (t,f, B^{s_0,q_0}_{p_0}, B^{s_1,q_1}_{p_1}, g)$ maps $B^{s_0,q_0}_{p_0} + B^{s_1,q_1}_{p_1}$ to $\mathbb{R}_{+}$
and $K_{\xi} (t,\{f_{j,\gamma}\}, \tilde{l}^{s_0,q_0}_{p_0}, \tilde{l}^{s_1,q_1}_{p_1}, \{g_{j,\gamma}\})$ maps $\tilde{l}^{s_0,q_0}_{p_0} + \tilde{l}^{s_1,q_1}_{p_1}$ to $\mathbb{R}_{+}$.
The above equation \eqref{eq:Kxi} transform the study of function spaces to which of sequence.
\end{proof}


Then we analyze the wavelet coefficient characteristics of approximate wavelet functional
and we turn the K functionals over the infinite complex field into the K functionals over the cuboid of the real field.
Let $ \mathfrak{C}^{\Lambda}_{f}= \{0\leq g_{j,\gamma}\leq |f_{j,\gamma}|\}_{(j,\gamma)\in \Lambda}$
denote an infinite cuboid  in $\mathbb{R}^{\Lambda}_{+}$.
For function $f$, using positive equivalence classes,
we transform the K functional to the infimum of positive sequence  on infinite cuboid
$ \mathfrak{C}^{\Lambda}_{f}$.
Let $(j,\gamma)\in \Lambda$.
If $f_{j,\gamma}\neq 0$,
let $(Bf)_{j,\gamma}= \frac{ f_{j,\gamma}}{|f_{j,\gamma}|}$.
If $f_{j,\gamma}= 0$, let $(Bf)_{j,\gamma}=1$.
Then $|(Bf)_{j,\gamma}|=1$ and $f_{j,\gamma}= |f_{j,\gamma}| (Bf)_{j,\gamma}.$
Let $\tilde{g}_{j,\gamma}= g_{j,\gamma}  (Bf)_{j,\gamma}$,
let $\tilde{g}=\{\tilde{g}_{j,\gamma}\}_{(j,\gamma)\in \Lambda}$.
In general, we can only conclude that
$\tilde{g}_{j,\gamma}\in \mathbb{C}^{\Lambda}.$
We call the following functional cuboid functional
$$K^{\mathfrak{C}}_{\xi} (t,f, \tilde{l}^{s_0,q_0}_{p_0}, \tilde{l}^{s_1,q_1}_{p_1})=
\inf \limits_{g_{j,\gamma}\in \mathfrak{C}^{\Lambda}_{f}}
K_{\xi} (t, Af, \tilde{l}^{s_0,q_0}_{p_0}, \tilde{l}^{s_1,q_1}_{p_1}, \tilde{g}).$$
$K^{\mathfrak{C}}_{\xi}$ denote the  functional on infinite cuboid,
so there is a function whose lower bound exists.
If $\xi=1$, denote further $K_{\mathfrak{C}}=K^{\mathfrak{C}}_{\xi}$.
The  functional $K^{\mathfrak{C}}_{\xi}$ on infinite cuboid equals the $K_{\xi}$ functional.
The following theorem transforms the study of K functional into the study of cuboid K functional.
\begin{theorem}\label{th:5.3}
Given $s_0,s_1\in \mathbb{R}, 0<p_0,p_1,q_0,q_1\leq \infty, 0<\xi<\infty$. We have,
\begin{equation}\label{eq:cuboid}
\begin{array}{rcl}
K_{\xi} (t,f, \tilde{l}^{s_0,q_0}_{p_0}, \tilde{l}^{s_1,q_1}_{p_1})
&=&K ^{\mathfrak{C}}_{\xi} (t, Af, \tilde{l}^{s_0,q_0}_{p_0}, \tilde{l}^{s_1,q_1}_{p_1}).
\end{array}
\end{equation}
\end{theorem}

\begin{proof}
(i)
Since $$|\tilde{g}_{j,\gamma}-f_{j,\gamma}|= |(\tilde{g}_{j,\gamma}  -f_{j,\gamma}) (Bf)_{j,\gamma}^{-1}|
= |g_{j,\gamma}  -f_{j,\gamma} (Bf)_{j,\gamma}^{-1}|
= |g_{j,\gamma}  -|f_{j,\gamma}||.$$

Hence $$K_{\xi} (t,\{f_{j, \gamma}\}, \tilde{l}^{s_0,q_0}_{p_0}, \tilde{l}^{s_1,q_1}_{p_1}, \{\tilde{g}_{j,\gamma}\})
= K_{\xi} (t,\{|f_{j, \gamma}|\}, \tilde{l}^{s_0,q_0}_{p_0}, \tilde{l}^{s_1,q_1}_{p_1}, \{g_{j,\gamma}\} ).$$

The K functional on the left-hand side of the above equation is defined by the lower limit of the sequences
$\{f_{j,\gamma}\}_{(j,\gamma)\in \Lambda} , \{g_{j,\gamma} \}_{(j,\gamma)\in \Lambda}$
which takes values on $\tilde{l}^{s_0,q_0}_{p_0}+ \tilde{l}^{s_1,q_1}_{p_1}\in \mathbb{C}^{\Lambda}$.
The K functional on the right-hand side of the above equation is defined by the lower limit of the sequences
$\{|f_{j,\gamma}|\}_{(j,\gamma)\in \Lambda} , \{g_{j,\gamma} \}_{(j,\gamma)\in \Lambda}$
where $\{|f_{j,\gamma}|\}_{(j,\gamma)\in \Lambda}$ is defined on $\tilde{l}^{s_0,q_0}_{p_0}+ \tilde{l}^{s_1,q_1}_{p_1}\in \mathbb{R}_{+}^{\Lambda}$
and $\{g_{j,\gamma}\}_{(j,\gamma)\in \Lambda}$ is defined on $\tilde{l}^{s_0,q_0}_{p_0}+ \tilde{l}^{s_1,q_1}_{p_1}\in \mathbb{C}^{\Lambda}$.

(ii) Since $|g_{j,\gamma}-|f_{j,\gamma}||
= \sqrt{||f_{j,\gamma}|- {\rm Re}\, g_{j,\gamma}|^{2} + ({\rm Im}\, g_{j,\gamma})^{2} }
\geq ||f_{j,\gamma}|- {\rm Re}\, g_{j,\gamma}|.$
Hence
$$K_{\xi} (t,\{|f_{j, \gamma}|\}, \tilde{l}^{s_0,q_0}_{p_0}, \tilde{l}^{s_1,q_1}_{p_1}, \{g_{j,\gamma}\} )
> K_{\xi} (t,\{|f_{j, \gamma}|\}, \tilde{l}^{s_0,q_0}_{p_0}, \tilde{l}^{s_1,q_1}_{p_1}, \{{\rm Re}\, g_{j,\gamma}\} ).$$
To consider the infimum, we need consider only for the case where $\{g_{j,\gamma}\}_{(j,\gamma)\in \Lambda}$ takes value in  $\mathbb{R}^{\Lambda}$.
Hence K functional defined by the infimum of $K_{\xi} (t,f, \tilde{l}^{s_0,q_0}_{p_0}, \tilde{l}^{s_1,q_1}_{p_1}, g_{j,\gamma})$
where the sequence $\{f_{j,\gamma}\}_{(j,\gamma)\in \Lambda}, \{g_{j,\gamma}\}_{(j,\gamma)\in \Lambda}$
take values on $\tilde{l}^{s_0,q_0}_{p_0}+ \tilde{l}^{s_1,q_1}_{p_1}\in \mathbb{C}^{\Lambda}$
can be transformed to the K functional defined by the infimum of sequence $\{|f_{j,\gamma}|\}_{(j,\gamma)\in \Lambda}$, $\{g_{j,\gamma}\}_{(j,\gamma)\in \Lambda}$
where $\{ |f_{j,\gamma}| \}_{(j,\gamma)\in \Lambda} $ is defined on  $\tilde{l}^{s_0,q_0}_{p_0}+ \tilde{l}^{s_1,q_1}_{p_1}\in \mathbb{R}_{+}^{\Lambda}$
and $\{g_{j,\gamma}\}_{(j,\gamma)\in \Lambda} $ is defined on $\tilde{l}^{s_0,q_0}_{p_0}+ \tilde{l}^{s_1,q_1}_{p_1}\in \mathbb{R}^{\Lambda}$.

(iii) Further, if $g_{j,\gamma}<0$, then
$$|f_{j,\gamma}|- g_{j,\gamma}=|f_{j,\gamma}|+|g_{j,\gamma}|> |f_{j,\gamma}|- |g_{j,\gamma}|.$$
Hence,
$$K (t,\{|f_{j, \gamma}|\}, \tilde{l}^{s_0,q_0}_{p_0}, \tilde{l}^{s_1,q_1}_{p_1}, \{g_{j,\gamma}\} )>
K (t,\{|f_{j, \gamma}|\}, \tilde{l}^{s_0,q_0}_{p_0}, \tilde{l}^{s_1,q_1}_{p_1}, \{|g_{j,\gamma}|\} ).$$
Hence we can restrict $\{g_{j,\gamma}\}_{(j,\gamma)\in \Lambda} \in \mathbb{R}^{\Lambda}$ to the case
$g_{j,\gamma}\geq 0$.

(iv) Fourthly, if $g_{j,\gamma}\geq 2|f_{j,\gamma}|$, then choose $g_{j,\gamma}^{0}=0$, we have
$$g_{j,\gamma}>g^0_{j,\gamma} \mbox { and } ||f_{j,\gamma}|- g_{j,\gamma}|\geq |f_{j,\gamma}|= |f_{j,\gamma}|- g^0_{j,\gamma}.$$
If $|f_{j,\gamma}|< g_{j,\gamma}<2|f_{j,\gamma}|$, then there exists $g^{0}_{j,\gamma}$ satisfying that
(i) $0<g^{0}_{j,\gamma}< |f_{j,\gamma}|$
and  (ii) $g^{0}_{j,\gamma}+g^{0}_{j,\gamma}=2|f_{j,\gamma}|$. Hence
$$g_{j,\gamma}>g^0_{j,\gamma} \mbox { and }
||f_{j,\gamma}|- g_{j,\gamma}|=g_{j,\gamma}-|f_{j,\gamma}| =|f_{j,\gamma}|-g^{0}_{j,\gamma}.$$
If $0\leq g_{j,\gamma}\leq |f_{j,\gamma}|$, then choose $g_{j,\gamma}^{0}=g_{j,\gamma}$. We have
$$K (t,\{|f_{j, \gamma}|\}, \tilde{l}^{s_0,q_0}_{p_0}, \tilde{l}^{s_1,q_1}_{p_1}, \{g_{j,\gamma}\} )>
K (t,\{|f_{j, \gamma}|\}, \tilde{l}^{s_0,q_0}_{p_0}, \tilde{l}^{s_1,q_1}_{p_1}, \{g^0_{j,\gamma}\} ).$$
Hence we can restrict $\{g_{j,\gamma}\}_{(j,\gamma)\in \Lambda} \in \mathbb{R}_{+}^{\Lambda}$
to the case $|f_{j,\gamma}|- g_{j,\gamma}\geq 0$, $\forall (j,\gamma)\in \Lambda$.
K functional defined by the infimum of sequence $\{f_{j,\gamma}\}_{(j,\gamma)\in \Lambda}, \{g_{j,\gamma}\}_{(j,\gamma)\in \Lambda}$
defined on $\tilde{l}^{s_0,q_0}_{p_0}+ \tilde{l}^{s_1,q_1}_{p_1}\in \mathbb{R}_{+}^{\Lambda}$
can be transformed to the K functional defined by the infimum of sequence $\{|f_{j,\gamma}|\}_{(j,\gamma)\in \Lambda}, \{g_{j,\gamma}\}_{(j,\gamma)\in \Lambda}$
where $\{|f_{j,\gamma}|\}_{(j,\gamma)\in \Lambda}$ is defined on $\tilde{l}^{s_0,q_0}_{p_0}+ \tilde{l}^{s_1,q_1}_{p_1}\in \mathbb{R}_{+}^{\Lambda}$
and $\{g_{j,\gamma}\}_{(j,\gamma)\in \Lambda}$  is defined on the infinite cuboid $\mathfrak{C}^{\Lambda}_{f}$.

Using positive equivalence classes and the above four steps,
we have transform the K functional to the infimum of positive sequence  on infinite cuboid
$ \mathfrak{C}^{\Lambda}_{f}$.
Hence we get \eqref{eq:cuboid}.

\end{proof}

The following theorem states that the wavelet basis is the unconditional basis of the Besov real interpolation space
and allows us to compare the sizes of K functionals by comparing the absolute values of the wavelet coefficients of the functions.
\begin{theorem}\label{th:absolute}
Given $s_0,s_1\in \mathbb{R}, 0<p_0,p_1,q_0,q_1\leq \infty$ and $0<\xi<\infty$.
If $|f_{j,\gamma}| \leq |g_{j,\gamma}|, \forall (j,\gamma)\in \Lambda$, then
$$K_{\xi} (t,f, \tilde{l}^{s_0,q_0}_{p_0}, \tilde{l}^{s_1,q_1}_{p_1})\leq K_{\xi} (t,g, \tilde{l}^{s_0,q_0}_{p_0}, \tilde{l}^{s_1,q_1}_{p_1}).$$
\end{theorem}

 Theorem \ref{th:absolute} tells us the $K_{\xi}$ functional is determined by the absolute value of the wavelet coefficients.
Hence denote $Af$ the function whose wavelet coefficients
are equal to the absolute value of the wavelet coefficients of $f$ and
$(Af)_{j,\gamma}=|f_{j,\gamma}|$.
According to  Theorems \ref{th:5.3} and \ref{th:absolute},
we only need to compute discrete K functionals with positive wavelet coefficients.
So we assume that all the wavelet coefficients are positive.

\section{Grid K functional} \label{sec:5}

We need to point out,
full grid $\Lambda$ is composed by main grid $\mathbb{N}$ and layer grid $\Gamma_{j}$.
$\mathbb{N}$ denote frequency stratification. For given $j\in \mathbb{N}$,
$\Xi_{j}$ denote the direction set.
Layer grid $\Gamma_{j}$  marks the direction and position at layer $j$.
In order to figure out the grid K functional,
we need to first figure out the discrete functional on main grid $\mathbb{N}$ and layer grid $\Gamma_{j}$.

In this section, we present some basic knowledge on lattices.
We consider some basic properties of $l^{p}$ real interpolations on the simple index set $\mathbb{N}$ and $\Gamma_{j}$ that make up the complex index set $\Lambda$.
The results in Sections \ref{sec:5.1} and \ref{sec:5.2} are corresponding to the generalization of the continuous case, we only state theorems without giving proofs.
In Section \ref{sec:8.1} we consider the K functionals when the function is restricted to a single layer and the classification on the corresponding index set.

\subsection{Main grid $\mathbb{N}$} \label{sec:5.1}
Consider first the layer grid $\mathbb{N}$. For $E\subset \mathbb{N}$, let $\#E$ denote the number of points in the set $\mathbb{N}$.
When $E$ is an empty set, denote  $\#E=0$.
Given $a=(a_j)_{j\in \mathbb{N}}$ where $a_j\geq 0$,
$b=(b_j)_{j\in \mathbb{N}}$ where $b_j\geq 0$.
Denote $E_{a}(\lambda) = \{ j\geq 0, a_{j}> \lambda\}, \forall \lambda>0$.
Denote the distribution function $\sigma_{a}(\lambda)= \# E_{a}(\lambda), \forall \lambda>0$ which maps $\mathbb{R}^{+} \rightarrow \mathbb{N}$.
Denote the non-increasing rearrangement function $a^{*}(t)= \inf \{\lambda:  \sigma_{a}(\lambda)\leq t\}, \forall t\in \mathbb{N}$ which maps $\mathbb{N} \rightarrow \mathbb{R}^{+}$.
Denote $$K(t,a, l^{q_0}, l^{q_1}, b) = \lf[\sum\limits_{j\geq 0} b_{j}^{q_0}\r]^{\frac{1}{q_0}}
+ t \lf[\sum\limits_{j\geq 0} (a_{j}-b_{j})^{q_1}\r]^{\frac{1}{q_1}}.$$
Denote K functional on the layer grid $\mathbb{N}$
$$K(t,a, l^{q_0}, l^{q_1}) = \inf\limits_{0\leq b_j\leq a_{j}} K(t,a, l^{q_0}, l^{q_1}, b).$$

In the following four cases,
we represent K functionals by rearrangement functions on the layer grid.
By similar skills in Lemma \ref{lem:2.5}, 
we have the following results.
\begin{lemma}\label{lem:5.1}
{\rm (i)} If $0<q_0<q_1=\infty$, then
$$K(t,a, l^{q_0}, l^{\infty}) = [ \int^{t^{q_0}}_{0} (a^{*}(\tau))^{q_0} d\tau ] ^{\frac{1}{q_0}}.$$

{\rm (ii)} If $0<q_0<q_1<\infty, \frac{1}{\beta}= \frac{1}{q_0}- \frac{1}{q_1}$, then
$$K(t,a, l^{q_0}, l^{q_1}) = [ \int^{t^{\beta}}_{0} (a^{*}(\tau))^{q_0} d\tau ] ^{\frac{1}{q_0}}
+ t [ \int^{\infty}_{t^{\beta}} (a^{*}(\tau))^{q_1} d\tau ] ^{\frac{1}{q_1}}.$$
\end{lemma}

If $q_0>q_1$ and $\frac{1}{\beta}= \frac{1}{q_1}- \frac{1}{q_0}$, then we have
$$\begin{array}{rcl} K(t,a, l^{q_0}, l^{q_1})&=& \inf (\|b\|_{l^{q_0}}+ t \|f_1\|_{l^{q_1}})
= t \inf (\|a-b\|_{l^{q_1}} + t^{-1} \|b\|_{l^{q_0}})\\
&=& t [ \int^{t^{-\beta}}_{0} (a^{*}(\tau))^{q_1} d\tau ] ^{\frac{1}{q_1}}
+ t^{-1} [ \int^{\infty}_{t^{-\beta}} (a^{*}(\tau))^{q_0} d\tau ] ^{\frac{1}{q_0}}\\
&=& [ \int^{\infty}_{t^{-\beta}} (a^{*}(\tau))^{q_0} d\tau ] ^{\frac{1}{q_0}}
+ t [ \int^{t^{-\beta}}_{0} (a^{*}(\tau))^{q_1} d\tau ] ^{\frac{1}{q_1}} .
\end{array}$$
Similar to the Lemma \ref{lem:5.1}, we have
\begin{lemma} \label{lem:5.2}
{\rm (i)} $0<q_1<q_0=\infty$, $K(t,a, l^{q_1}, l^{\infty}) = [ \int^{t^{-q_1}}_{0} (a^{*}(\tau))^{q_1} d\tau ] ^{\frac{1}{q_1}}$.

{\rm (ii)} $0<q_1<q_0<\infty, \frac{1}{\beta}= \frac{1}{q_1}- \frac{1}{q_0}$,
$$K(t,a, l^{q_0}, l^{q_1}) = [ \int^{\infty}_{t^{-\beta}} (a^{*}(\tau))^{q_0} d\tau ] ^{\frac{1}{q_0}}
+ t [ \int^{t^{-\beta}}_{0} (a^{*}(\tau))^{q_1} d\tau ] ^{\frac{1}{q_1}}.$$
\end{lemma}


\subsection{Layer grid $\Gamma_{j}$ and classification of vertex} \label{sec:5.2}

For any $j\geq 0$, $E\subset \Gamma_{j}$ is a discrete set and let $\# E$ denote the number of points $\gamma$ in $E\subset \Gamma_{j}$.
When $E$ is an empty set, denote  $\#E=0$.
Hence $|\{x: f^{\infty}_{j}(x)> 2^{u}\}|= 2^{-nj} \#\{ \gamma\in \Gamma_{j}, 2^{\frac{nj}{2}} |f_{j,\gamma}|>2^{u}\}$.
$\forall \lambda>0$,
denote $E_{f_{j}}(\lambda) = \{ \gamma\in \Gamma_{j}, |f_{j,\gamma}| >\lambda\}$.
and $\sigma_{f_{j}}(\lambda)= \# E_{f_{j}}(\lambda)$ is a function which maps from $\mathbb{R}^{+}$ to $\mathbb{N}$.
$\forall t\in \mathbb{N}, f^{*}_{j}(t)= \inf \{\lambda:  \sigma_{f_{j}}(\lambda)\leq t\}$
is a function maps from $\mathbb{N}$ to $\mathbb{R}^{+}$.
We assume $f_{j,\gamma}=|f_{j,\gamma}|$ and denote
\begin{equation*}
\begin{array}{l}
K(t,f_{j}, l^{p_0}, l^{p_1})
= \inf\limits_{0\leq g_{j,\gamma}\leq |f_{j,\gamma}|}
\{  (\sum\limits_{\gamma\in \Gamma_j} g_{j,\gamma}^{p_0})^{\frac{1}{p_0}}
+t (\sum\limits_{\gamma\in \Gamma_j} |f_{j,\gamma}- g_{j,\gamma}|^{p_1})^{\frac{1}{p_1}}\}.
\end{array}
\end{equation*}

We need to classify the vertices of cuboids according to the above equations \eqref{eq:(8)} and \eqref{eq:(9)}
in the above Lemma \ref{th:CDL}.
We have to use the notations of
$\Gamma_{j}^{0}(t)$ and $\Gamma_{j}^{1}(t)$
throughout the rest of this paper.
{\bf Case 1.} $0<p_0<p_1=\infty$.
For $\sigma_{f_{j}}(\lambda) = t^{p_0}$, denote $\Gamma_{j}^{0}(t) = \{\gamma: |f_{j,\gamma}|>\lambda\},
\Gamma_{j}^{1}(t) = \{\gamma: \gamma\notin \Gamma_{j}^{0}(t)\}$.
\begin{lemma}\label{lem:6.4}
Given $0<p_0<\infty$. Then
\begin{equation}\label{eq:WG1}
K(t,f_{j}, l^{p_0}, l^{\infty}) = \lf[\int^{t^{p_0}}_{0} (f^{*}_{j}(\tau))^{p_0} d\tau\r]^{\frac{1}{p_0}}
= \lf(\sum\limits_{\gamma\in \Gamma^{0}_{j}(t)} f^{p_0}_{j,\gamma} \r)^{\frac{1}{p_0}}.
\end{equation}
\end{lemma}

{\bf Case 2.} $0<p_0<p_1<\infty$ and $\frac{1}{\alpha}=\frac{1}{p_0}- \frac{1}{p_1}$.
For $\sigma_{f_{j}}(\lambda) = t^{\alpha}$, denote $\Gamma_{j}^{0}(t) = \{\gamma: |f_{j,\gamma}|>\lambda\},
\Gamma_{j}^{1}(t) = \{\gamma: \gamma\notin \Gamma_{j}^{0}(t)\}$.
\begin{lemma} \label{lem:6.5}
Given $0<p_0<p_1<\infty$ and $\frac{1}{\alpha}=\frac{1}{p_0}- \frac{1}{p_1}$. Then
\begin{equation}\label{eq:WG2}
\begin{array}{rcl}
K(t,f_{j}, l^{p_0}, l^{p_1}) &=& \lf[\int^{t^{\alpha}}_{0} (f^{*}_{j}(\tau))^{p_0} d\tau\r]^{\frac{1}{p_0}}
+ t\lf[\int^{\infty}_{t^{\alpha}} (f^{*}_{j}(\tau))^{p_1} d\tau\r]^{\frac{1}{p_1}}\\
&=& \lf[\sum\limits_{\gamma\in \Gamma^{0}_{j}(t)} f^{p_0}_{j,\gamma} \r]^{\frac{1}{p_0}}
+ t\lf[\sum\limits_{\gamma\in \Gamma^{1}_{j}(t)} f^{p_1}_{j,\gamma} \r]^{\frac{1}{p_1}}.
\end{array}
\end{equation}
\end{lemma}


\subsection{Single layer K functional}\label{sec:8.1}
In this subsection, we consider K functionals corresponding to real interpolation of Besov spaces restricted to a single layer
and discuss vertex classification in a single layer.
For $s_0,s_1\in \mathbb{R}, 0<p_0, p_1, q_0, q_1 \leq \infty$,
denote $\tilde{s}= s_1-s_0+ \frac{n}{p_0}-\frac{n}{p_1}$ from here in this paper.
For $j\geq 0$, denote
\begin{equation}\label{eq:6.30000}
f_{j}=(f_{j,\gamma})_{\gamma\in \Gamma_j} \mbox{ and }
\tilde{f}_{j}=(|f_{j,\gamma}|)_{\gamma \in \Gamma_j}.
\end{equation}
For $0<\xi<\infty$, we have
\begin{theorem}\label{th:single}
Given $s_0,s_1\in \mathbb{R}, 0<p_0, p_1, q_0, q_1 \leq \infty$ and $0<\xi<\infty$.

{\rm (i)} If $p_0=p_1=p$, then
$$K_{\xi} (t,f_{j}, B^{s_0,q_0}_{p_0}, B^{s_1,q_1}_{p_1})\sim
2^{j(s_0 +\frac{n}{2}- \frac{n}{p_0})} \min (1, t2^{j\tilde{s}})
\lf(\sum\limits_{\gamma\in \Gamma_j} |f_{j,\gamma}| ^{p}\r) ^{\frac{1}{p}}.$$

{\rm (ii)} If $p_0\neq p_1$, then
$$K_{\xi} (t,f_{j}, B^{s_0,q_0}_{p_0}, B^{s_1,q_1}_{p_1})\sim K(t,f_{j}, B^{s_0,q_0}_{p_0}, B^{s_1,q_1}_{p_1})=
2^{j(s_0 +\frac{n}{2}- \frac{n}{p_0})} K(t2^{j\tilde{s}}, \tilde{f}_{j}, l^{p_0}, l^{p_1}).$$
\end{theorem}

\begin{proof}
According to the equation \eqref{eq:Kxi},
\begin{equation*}
\begin{array}{rl}
& K_{\xi} (t,f_j, B^{s_0,q_0}_{p_0}, B^{s_1,q_1}_{p_1}, g_j)\\
=&
\{ 2^{j \xi (s_0+ \frac{n}{2}-\frac{n}{p_0})}
(\sum\limits_{\gamma\in \Gamma_j} g_{j,\gamma}^{p_0}) ^{\frac{\xi}{p_0}}
+ t^{\xi}  2^{j \xi (s_1+ \frac{n}{2}-\frac{n}{p_1})}
(\sum\limits_{\gamma\in \Gamma_j} (g_{j,\gamma}-f_{j,\gamma}) ^{p_1}) ^{\frac{\xi}{p_1}}   \}^{\frac{1}{\xi}}\\
\sim &
2^{j  (s_0+ \frac{n}{2}-\frac{n}{p_0})}
(\sum\limits_{\gamma\in \Gamma_j} g_{j,\gamma}^{p_0}) ^{\frac{1}{p_0}}
+ t  2^{j (s_1+ \frac{n}{2}-\frac{n}{p_1})}
(\sum\limits_{\gamma\in \Gamma_j} (g_{j,\gamma}-f_{j,\gamma}) ^{p_1}) ^{\frac{1}{p_1}}   \\
\equiv &
K (t,f_{j}, \tilde{l}^{s_0,q_0}_{p_0}, \tilde{l}^{s_1,q_1}_{p_1}, g_{j}).
\end{array}
\end{equation*}

It is easy to see that (i) is true.
Furthermore, if $p_0\neq p_1$, then
\begin{equation*} \label{eq:singlelayer}
\begin{array}{l}
K(t,f_{j}, \tilde{l}^{s_0,q_0}_{p_0}, \tilde{l}^{s_1,q_1}_{p_1})\\
\sim 2^{j(s_0 +\frac{n}{2}- \frac{n}{p_0})}
\inf\limits_{0\leq g_{j,\gamma}\leq |f_{j,\gamma}|}
\lf\{  (\sum\limits_{\gamma} g_{j,\gamma}^{p_0})^{\frac{1}{p_0}}
+t 2^{j\tilde{s}} (\sum\limits_{\gamma} ||f_{j,\gamma}|- g_{j,\gamma}|^{p_1})^{\frac{1}{p_1}}\r\}\\
= 2^{j(s_0 +\frac{n}{2}- \frac{n}{p_0})} K(t2^{j\tilde{s}},\tilde{f}_{j}, l^{p_0}, l^{p_1}).
\end{array}
\end{equation*}

\end{proof}

Theorem \ref{th:single}
provides a classification method of wavelet coefficients for the case of single-layer K functional for $p_0\neq p_1$.
By Lemma \ref{lem:CK} and by similarity, we just have to think about $0<p_0<p_1\leq \infty$.
Let $f^{*}_{j}$ be the relative function defined in the above section \ref{sec:5.2}.
\begin{theorem}\label{th:6.60000}
Given $s_0,s_1\in \mathbb{R}, 0< q_0, q_1 \leq \infty$.

(i) For $0<p_0<p_1=\infty$, by equation \eqref{eq:WG1} in Lemma \ref{lem:6.4}, we have
\begin{equation}\label{eq:666}
\begin{array}{rcl}
K(t,f_{j}, \tilde{l}^{s_0,q_0}_{p_0}, \tilde{l}^{s_1,q_1}_{\infty})  &= &2^{j(s_0 +\frac{n}{2}- \frac{n}{p_0})} \lf[\int^{(t2^{j\tilde{s}})^{p_0}}_{0} (f^{*}_{j}(\tau))^{p_0} d\tau\r]^{\frac{1}{p_0}}\\
&=& 2^{j(s_0 +\frac{n}{2}- \frac{n}{p_0})} \lf(\sum\limits_{\gamma\in \Gamma^{0}_{j}(t2^{j\tilde{s}})} f^{p_0}_{j,\gamma} \r)^{\frac{1}{p_0}}.
\end{array}
\end{equation}

(ii) For $0<p_0<p_1<\infty$, by equation \eqref{eq:WG2} in Lemma \ref{lem:6.5}, we have
\begin{equation}\label{eq:667}
\begin{array}{rl}
&K(t,f_{j}, \tilde{l}^{s_0,q_0}_{p_0}, \tilde{l}^{s_1,q_1}_{p_1}) \\
=& 2^{j(s_0 +\frac{n}{2}- \frac{n}{p_0})}
\lf[\int^{(t2^{j\tilde{s}})^{\alpha}}_{0} (f^{*}_{j}(\tau))^{p_0} d\tau\r]^{\frac{1}{p_0}}
+ t 2^{j(s_1 +\frac{n}{2}- \frac{n}{p_1})} \lf[\int^{\infty}_{(t2^{j\tilde{s}})^{\alpha}} (f^{*}_{j}(\tau))^{p_1} d\tau\r]^{\frac{1}{p_1}}\\
=& 2^{j(s_0 +\frac{n}{2}- \frac{n}{p_0})} \lf(\sum\limits_{\gamma\in \Gamma^{0}_{j}(t2^{j\tilde{s}})} f^{p_0}_{j,\gamma} \r)^{\frac{1}{p_0}}
+ t2^{j(s_1 +\frac{n}{2}- \frac{n}{p_1})}\lf (\sum\limits_{\gamma\in \Gamma^{1}_{j}(t2^{j\tilde{s}})} f^{p_1}_{j,\gamma} \r)^{\frac{1}{p_1}}.
\end{array}
\end{equation}
\end{theorem}

Later we will use this above classification to consider the corresponding K functional.
In Sections \ref{sec:7} and \ref{sec:8}, we use cuboid functional.
In Section \ref{sec:7}, we obtain the concrete wavelet characterization of K functional for the case $p_0=p_1$.
In Section \ref{sec:8}, we obtain the concrete wavelet characterization of K functional for the case $q_0=q_1$ and $p_0<p_1$.
Cuboid functional is not enough to consider more complex indices.
In Section \ref{sec:9}, we first narrow down the K functionals further to the vertex K functionals.
In Section \ref{sec:9b}, we apply vertex K functional to get a new concrete wavelet characterization of K functional for $q_0=q_1$.
In Section \ref{sec:10}, we consider the case $p_0 \neq p_1.$
We introduce vertex $K_{\infty}$ functional, power spaces and topology compatiability.
We get finally the concrete wavelet characterization of K functional for $p_0\neq p_1$ and $q_0\neq q_1$.

\section{K functional when $p_0=p_1$} \label{sec:7}

For $p_0=p_1=p$,
the real interpolation spaces
$(B^{s_0, q_0}_{p}, B^{s_1, q_1}_{p})_{\theta, r}$
are Besov spaces for the cases where (i) $s_0\neq s_1$ and (ii) $s_0=s_1$ and $r=q$.
For the cases where $s_0=s_1$ and $r\neq q$,
no one had considered it before.
We will obtain the relative concrete wavelet characterization of K functional and prove that the corresponding real interpolation space runs out of the Besov space series.
For $s_0\neq s_1$,
to the best of our knowledge,
no one has systematically given a concrete expression for the K functional before.
Here we systematically give the specific wavelet characterization of K functional
$K_{\xi} (t,f, B^{s_0,q_0}_{p}, B^{s_1,q_1}_{p})$.

In section \ref{sec:6.2}, we use cuboid K functionals to convert the computation of these functionals
into real interpolation functionals
with respect to $f^{p}$ functions at lattice points $\mathbb{N}$.
In the remaining three sections, we compute the K functional in the following three cases:
(i) $q_0=q_1$; (ii) $s_0\neq s_1$ and $q_0\neq q_1$; (iii) $s_0=s_1$.
We calculate the corresponding K functional,
and get the specific form of the corresponding interpolation space and the relationship with Besov spaces.

\subsection{Transformation of K functional when $p$ is equal} \label{sec:6.2}
When $p_0=p_1=p$, we solidify the norm on the layer grid $\Gamma_j$ and transform the K functional on the main grid.
{\bf We transform the study of K functionals into the study of K functional
for weighted norm $l^{s,q}$ of sequences $f^{p}$ over main grid $\mathbb{N}$.}
Note that
\begin{equation*}
\begin{array}{rcl}  K_{\xi} (t,f, B^{s_0,q_0}_{p}, B^{s_1,q_1}_{p}, g) &=&
\{ [\sum\limits_{j\geq 0} 2^{j q_0 (s_0+ \frac{n}{2}-\frac{n}{p})}
(\sum\limits_{\gamma} g_{j,\gamma}^{p}) ^{\frac{q_0}{p}} ] ^{\frac{\xi}{q_0}}\\
&&+ t^{\xi} [\sum\limits_{j\geq 0} 2^{j q_1 (s_1+ \frac{n}{2}-\frac{n}{p})}
(\sum\limits_{\gamma} (g_{j,\gamma}-f_{j,\gamma}) ^{p}) ^{\frac{q_1}{p}} ] ^{\frac{\xi}{q_1}}\}^{\frac{1}{\xi}}
\end{array}
\end{equation*}
When $p_0=p_1=p$, we transform the K functional for Besov spaces to the K functional for layer sequence.
We may assume that $f_{j,\gamma}=|f_{j,\gamma}|$.

{\bf Case 1}. If $p_0=p_1= p=\infty$, then $\forall j\in \mathbb{N}$, denote
\begin{equation}\label{eq:7.10000}
f^{\infty}_{j}= \sup\limits_{\gamma} |f_{j,\gamma}|, f^{\infty} = (f^{\infty}_{j})_{j\geq 0}, g^{\infty}_{j}= \sup\limits_{\gamma} |g_{j,\gamma}|.
\end{equation}
Hence $f^{\infty}_{j}= g^{\infty}_{j}+ f^{\infty}_{j}-g^{\infty}_{j}$. Since $g_{j,\gamma}\in \mathfrak{C}^{\Lambda}_{f}$, we have
$$0\leq \sup\limits_{\gamma} (|f_{j,\gamma}|- g_{j,\gamma}) \leq f^{\infty}_{j} \leq \sup\limits_{\gamma} g_{j,\gamma}
+ \sup\limits_{\gamma} (|f_{j,\gamma}|- g_{j,\gamma}).$$
For $f^{\infty}$ defined in the equation \eqref{eq:7.10000},  we have
\begin{lemma} \label{lem:7.1}
Given $s_0,s_1\in \mathbb{R}, 0<q_0,q_1\leq \infty$ and $0<\xi<\infty$. Then
\begin{equation*}
\begin{array}{rl} &  K_{\xi} (t,f, B^{s_0,q_0}_{\infty}, B^{s_1,q_1}_{\infty})\\
\sim & \inf \limits_{0\leq g_{j}\leq f^{\infty}_{j}}
\{[\sum\limits_{j\geq 0} 2^{j q_0 (s_0+ \frac{n}{2})}
(g^{\infty}_{j}) ^{q_0} ] ^{\frac{\xi}{q_0}} + t^{\xi} [\sum\limits_{j\geq 0} 2^{j q_1 (s_1+ \frac{n}{2})}
(f^{\infty}_{j}-g^{\infty}_{j}) ^{q_1} ] ^{\frac{\xi}{q_1}}\}^{\frac{1}{\xi}}\\
=& \inf \limits_{0\leq g_{j}\leq f^{\infty}_{j}} K_{\xi} (t, f^{\infty}_{j} , l^{s_0+ \frac{n}{2},q_0}, l^{s_1+ \frac{n}{2},q_1},g^{\infty}_{j} )
= K_{\xi} (t, f^{\infty}, l^{s_0+ \frac{n}{2},q_0}, l^{s_1+ \frac{n}{2},q_1}).
\end{array}
\end{equation*}
\end{lemma}

{\bf Case 2}. If $0<p=p_0=p_1 <\infty$, then denote
\begin{equation}\label{eq:7.20000}
f^{p}_{j}= (\sum\limits_{\gamma} |f_{j,\gamma}|^{p})^{\frac{1}{p}}, f^{p} = (f^{p}_{j})_{j\geq 0}.
\end{equation}
Hence we have $0 \leq (\sum\limits_{\gamma} |g_{j,\gamma}|^{p})^{\frac{1}{p}} \leq f^{p}_{j}$ and
$$0\leq [ \sum\limits_{\gamma} (f_{j,\gamma}- g_{j,\gamma})^{p}]^{\frac{1}{p}}
\leq f^{p}_{j} \leq (\sum\limits_{\gamma} |g_{j,\gamma}|^{p})^{\frac{1}{p}}
+ [\sum\limits_{\gamma} (f_{j,\gamma}- g_{j,\gamma})^{p}]^{\frac{1}{p}}.$$
Further, we can take $g_{j\gamma}= g^{p}_{j} \frac{f_{j,\gamma}}{f^{p}_{j}}$ and
$(\sum\limits_{\gamma} |g_{j,\gamma}|^{p})^{\frac{1}{p}}= g^{p}_{j}$,
$(\sum\limits_{\gamma} (f_{j,\gamma}- g_{j,\gamma})^{p})^{\frac{1}{p}}= f^{p}_{j}- g^{p}_{j}$.
For $f^{p}$ defined in the equation \eqref{eq:7.20000}, we have

\begin{lemma} \label{lem:lay}
Given $s_0,s_1\in \mathbb{R}, 0<q_0,q_1\leq \infty$ and $ 0<\xi<\infty$.
\begin{equation*}
\begin{array}{rl} &  K_{\xi} (t,f, B^{s_0,q_0}_{p}, B^{s_1,q_1}_{p}) \\
\sim & \{ \inf \limits_{0\leq g^{p}_{j}\leq f^{p}_{j}}
[\sum\limits_{j\geq 0} 2^{j q_0 (s_0+ \frac{n}{2}-\frac{n}{p})}
(g^{p}_{j}) ^{q_0} ] ^{\frac{\xi}{q_0}} + t^{\xi} [\sum\limits_{j\geq 0} 2^{j q_1 (s_1+ \frac{n}{2}-\frac{n}{p})}
(f^{p}_{j}-g^{p}_{j}) ^{q_1} ] ^{\frac{\xi}{q_1}}\}^{\frac{1}{\xi}}\\
=&  \inf \limits_{0\leq g^{p}_{j}\leq f^{p}_{j}} K_{\xi} (t, f^{p}_{j}, l^{s_0+ \frac{n}{2}-\frac{n}{p},q_0}, l^{s_1+ \frac{n}{2}-\frac{n}{p},q_1},g^{p}_{j} )\\
=& K_{\xi} (t,f^{p}, l^{s_0+ \frac{n}{2}-\frac{n}{p},q_0}, l^{s_1+ \frac{n}{2}-\frac{n}{p},q_1}).
\end{array}
\end{equation*}
\end{lemma}

For the above two cases, we consider in particular the case where $s_0= s_1=s$.
We transform the study of K functionals on weighted space $l^{s,q}$ into the study of K functionals on unweighted space $l^{q}$.
In fact, for $0<p\leq \infty$ and for $f^{p}$ defined in the equations \eqref{eq:7.10000} and \eqref{eq:7.20000},
we set the weight function and denote
\begin{equation}\label{eq:7.30000}
f^{s,p}_{j} = 2^{j(s+\frac{n}{2}-\frac{n}{p})} f^{p}_{j}, f^{s,p}= (f^{s,p}_{j})_{j\geq 0}, g^{s,p}= (g^{s,p}_{j})_{j\geq 0}.
\end{equation}
Hence for $f^{s,p}$ defined in the equation \eqref{eq:7.30000}, by Lemmas \ref{lem:7.1} and \ref{lem:lay},
K functionals become $l^{q}$ functionals  of the Sobolev norm $W^{s,p}$ over the ring:
\begin{corollary}\label{lem:7.3}
Given $s_0=s_1=s\in \mathbb{R}, 0<p, q_0\neq q_1\leq \infty$ and $0<\xi<\infty$. Then
\begin{equation*}
\begin{array}{rcl}  K_{\xi} (t,f, B^{s,q_0}_{p}, B^{s,q_1}_{p}) &\sim & \inf \limits_{0\leq g^{s,p}_{j}\leq f^{s,p}_{j}}
\{[\sum\limits_{j\geq 0} (g^{s,p}_{j}) ^{q_0} ] ^{\frac{\xi}{q_0}} + t^{\xi} [\sum\limits_{j\geq 0}
(f^{s,p}_{j}-g^{s,p}_{j}) ^{q_1} ] ^{\frac{\xi}{q_1}}\}^{\frac{1}{\xi}}\\
&=& \inf \limits_{0\leq g^{s,p}_{j}\leq f^{s,p}_{j}} K_{\xi} (t,f^{s,p}, l^{q_0}, l^{q_1},g^{s,p} ) =
K_{\xi} (t, f^{s,p}, l^{q_0}, l^{q_1})
\end{array}
\end{equation*}
\end{corollary}

For the rest of this section, we calculate $K_{\xi} (t,f^{p}, l^{s_0+ \frac{n}{2}-\frac{n}{p},q_0}, l^{s_1+ \frac{n}{2}-\frac{n}{p},q_1})$
and $K_{\xi} (t,f^{s,p}, l^{q_0}, l^{q_1},g^{s,p} )$ in three cases.

\subsection{K functional where $p_0=p_1,q_0=q_1$  and $s_0\neq s_1$} \label{sec:7.1}
Let $0<q_0=q_1=q \leq \infty$ and $a, b\in \mathbb{R}$. 
We know
\begin{equation}\label{eq:7ps}
K (t,f, l^{a,q}, l^{b,q})= \inf \limits_{0\leq g_{j}\leq f_{j}}
\lf(\sum\limits_{j\geq 0} 2^{jaq} g_{j}^{q}\r)^{\frac{1}{q}}
+ t\lf(\sum\limits_{j\geq 0} 2^{jbq} (f_{j}- g_{j})^{q}\r)^{\frac{1}{q}}.
\end{equation}
We transform the K functional in \eqref{eq:7ps} into vertex K functional: a classification problem of layers.
In fact, let $\mathbb{N}_{V}$ be a subset of $\mathbb{N}$ and let $\mathbb{N}^{c}_{V}$ be the relative complementary set.
Denote
\begin{equation}\label{eq:nNV}
K _{\mathbb{N}_{V}} (t,f, l^{a,q}, l^{b,q})=
\lf(\sum\limits_{j\in \mathbb{N}_{V}} 2^{jaq} f_{j}^{q}\r)^{\frac{1}{q}}
+ t\lf(\sum\limits_{j\in \mathbb{N}^{c}_{V}} 2^{jbq} f_{j}^{q}\r)^{\frac{1}{q}}.
\end{equation}
We consider the lower bound of $K _{\mathbb{N}_{V}} (t,f, l^{a,q}, l^{b,q})$ in the equation \eqref{eq:nNV} and define
\begin{equation}\label{eq:NV}
K _{V} (t,f, l^{a,q}, l^{b,q})= \inf \limits_{\mathbb{N}_{V} }
K _{\mathbb{N}_{V}} (t,f, l^{a,q}, l^{b,q}).
\end{equation}
For $f_{j}>0$ in the equation \eqref{eq:7ps},
we distinguish the cases where $g_{j}\geq \frac{1}{2} f_{j}$ and $g_{j}< \frac{1}{2} f_{j}$.
$K (t,f, l^{a,q}, l^{b,q})$ in \eqref{eq:7ps} is equivalent to vertex K functional $K _{V} (t,f, l^{a,q}, l^{b,q})$ in \eqref{eq:NV}.
\begin{lemma}\label{lem:7.4}
Given $a,b \in \mathbb{R}$ and $0<q\leq \infty$.  We have
$$K (t,f, l^{a,q}, l^{b,q})\sim K _{V} (t,f, l^{a,q}, l^{b,q}).$$
\end{lemma}

Then we give the specific expression of the vertex K functional by the classification of $j$.
Fix $a<b$ and $0<q\leq \infty$. If $t 2^{j(b-a)}>1$, then denote $j\in \mathbb{N}_{t}.$
Otherwise, denote $j\in \mathbb{N}^{c}_{t}.$ Define
\begin{equation} \label{eq:basic}
W  (t,f, l^{a,q}, l^{b,q})= \lf(\sum\limits_{j\in \mathbb{N}_{t}} 2^{jaq} f_{j}^{q}\r)^{\frac{1}{q}}
+ t\lf(\sum\limits_{j\notin \mathbb{N}_{t}} 2^{jbq} f_{j}^{q}\r)^{\frac{1}{q}}.
\end{equation}
Combine the equations \eqref{eq:NV} and \eqref{eq:basic}, by Lemma \ref{lem:7.4}, we have
\begin{lemma} \label{lem:q}
Given $a<b$ and $0<q\leq \infty$.  We have
$$ K (t,f, l^{a,q}, l^{b,q}) \sim  W (t,f, l^{a,q}, l^{b,q}).$$
\end{lemma}

\begin{proof}
It is sufficient to prove $$ K _{V} (t,f, l^{a,q}, l^{b,q})= W (t,f, l^{a,q}, l^{b,q}).$$
(i) By definition,
$ K_{V} (t,f, l^{a,q}, l^{b,q})\leq  W (t,f, l^{a,q}, l^{b,q}).$\\
(ii) In turn, we prove
$$K_{\mathbb{N}_{V}}\geq K_{\mathbb{N}_{V}\bigcap \mathbb{N}_{t}}\geq K_{\mathbb{N}_{t}}.$$
In fact, if there exists $j_0\in \mathbb{N}_{V}$ such that $j_0\notin \mathbb{N}_{t}$,
then we denote $\mathbb{N}^{j_0,-}_{V}= \{j\in \mathbb{N}_{V} \mbox {\, and \,} j\neq j_0\}$.
It is easy to see that
$$K _{\mathbb{N}^{j_0,-}_{V}} (t,f, l^{a,q}, l^{b,q})< K _{\mathbb{N}_{V}} (t,f, l^{a,q}, l^{b,q}).$$
We continue this process and we get
$K_{\mathbb{N}_{V}}\geq K_{\mathbb{N}_{V}\bigcap \mathbb{N}_{t}}$.

If there exists $j_0\in \mathbb{N}_{t}$ and $j_0\notin \mathbb{N}_{V}\bigcap \mathbb{N}_{t}$,
then we denote $\mathbb{N}^{j_0}_{V}= \{j\in \mathbb{N}_{V} \mbox {\, or \,} j= j_0\}$.
It is easy to see that
$$K^{j_0} _{\mathbb{N}_{V}} (t,f, l^{a,q}, l^{b,q})< K _{\mathbb{N}_{V}\bigcap \mathbb{N}_{t}} (t,f, l^{a,q}, l^{b,q}).$$
We continue this process and we get
$K_{\mathbb{N}_{V}\bigcap \mathbb{N}_{t}}\geq K_{\mathbb{N}_{t}} $.

\end{proof}

Given $0<\theta_0<1$, $0<q_0,  q\leq \infty$,
$a<b$ and $s_0= a(1-\theta_0) + b\theta_0 $.
Applying the K functional obtained above,
we can obtain the real interpolation spaces of the corresponding weighted spaces
and we can obtain further the real interpolation spaces and K functional of the corresponding Besov spaces.
For weighted discrete Lebesgue spaces $l^{a,q}$ and $l^{b,q}$, by Lemma \ref{lem:q}, we have
\begin{theorem} \label{lem:7.6}
Given $0<\theta_0<1$, $0<q_0,  q\leq \infty$,
$a<b$ and $s_0= a(1-\theta_0) + b\theta_0 $.
We have
$$(l^{a,q}, l^{b,q})_{\theta_0, q_0}= l^{a(1-\theta_0)+ b\theta_0, q_0} = l^{s_0,q_0}.$$
\end{theorem}

According to the layer language Lemmas \ref{lem:7.1} and \ref{lem:lay}, the translation of the above Theorem \ref{lem:7.6}
says that the relative real interpolation spaces is still Besov spaces.
\begin{corollary} \label{cor:7.7}
Given $0<\theta_0<1$, $0<p, q_0,  q\leq \infty$,
$a<b$ and $s_0= a(1-\theta_0) + b\theta_0 $. Then
$$(B^{a,q}_{p}, B^{b,q}_{p})_{\theta_0, q_0}= B^{a(1-\theta_0)+ b\theta_0, q_0}_{p} = B^{s_0,q_0}_{p}.$$
$$K(t,f, B^{a,q}_{p}, B^{b,q}_{p})=  W(t,f^{p}, l^{a+\frac{n}{2}-\frac{n}{p},q}, l^{b+\frac{n}{2}-\frac{n}{p},q}), $$
where $f^{p}$ is defined in the equations \eqref{eq:7.10000} and \eqref{eq:7.20000} by wavelet coefficients,
$W(t,f^{p}, l^{a,q}, l^{b,q})$ is defined by wavelet coefficients in \eqref{eq:basic}.
\end{corollary}

\subsection{K functional where $p_0=p_1$  and $s_0\neq s_1, q_0\neq q_1$} \label{sec:7.11}
We consider then $q_0\neq q_1$ and $s_0\neq s_1$.
For these cases, we can use the result that $q_0=q_1=q$ and standard techniques like the {\bf Lions-Peetre Iterative Theorem}.
Given $0<\theta_0, \theta_1, \theta<1$, $0<q_0, q_1, q\leq \infty$, $0<r\leq \infty$, $s_0\neq s_1, \theta_0\neq \theta_1$,
$a= \frac{s_0 \theta_1 - s_1 \theta_0} {\theta_1 -\theta_0}$, $b= \frac{s_1 (1-\theta_0) - s_0 (1- \theta_1)}{\theta_1 -\theta_0}$.
Hence $a(1-\theta_0) + b\theta_0 = s_0$ and $a(1-\theta_1) + b\theta_1 = s_1$.
For weighted discrete Lebesgue spaces $l^{a,q}$ and $l^{b,q}$,
by applying Lions-Peetre's Iterative Theorem, we have
\begin{theorem} \label{lem:layin}
If $s_0\neq s_1$, $0<q_0\neq q_1\leq \infty$, $0<\theta<1$ and $0<r\leq \infty$, then
$$(l^{s_0,q_0}, l^{s_1,q_1})_{\theta,r}= l^{s_0(1-\theta )+ s_1 \theta, r}.$$
\end{theorem}

\begin{proof}
By Theorem \ref{lem:7.6}, for $a<b$,  we have
$$(l^{a,q}, l^{b,q})_{\theta_0, q_0}= l^{a(1-\theta_0)+ b\theta_0, q_0} = l^{s_0,q_0}.$$
$$(l^{a,q}, l^{b,q})_{\theta_1, q_1}= l^{a(1-\theta_1)+ b\theta_1, q_1} = l^{s_1,q_1}.$$
By applying Lemma \ref{lem:LP}, we get the conclusion of the above Theorem \ref{lem:layin}.

\end{proof}

Given $0<\theta_0<\theta_1<1$, $r= \theta_1-\theta_0$ and $0<q_0,q_1\leq \infty$.
By Holmstedt Theorem and Theorem \ref{lem:layin}, applying the notation in the equation \eqref{eq:basic}, we have
\begin{theorem} \label{lem:q01}
If $s_0\neq s_1$ and $0<q_0\neq q_1\leq \infty$, then
\begin{equation} \label{eq:basic2}
\begin{array}{l}
K(t,f, l^{s_0,q_0}, l^{s_1,q_1})\\
=  ( \int^{t^{\frac{1}{r}}}_{0} (s^{-\theta_0} W(s, f, l^{a,q},l^{b,q}))^{q_0} \frac{ds}{s})^{\frac{1}{q_0}}
+ t ( \int^{\infty}_{t^{\frac{1}{r}}} (s^{-\theta_1} W(s, f, l^{a,q},l^{b,q}))^{q_1} \frac{ds}{s})^{\frac{1}{q_1}}.
\end{array} \end{equation}
\end{theorem}

When $p_0=p_1=p$, $s_0\neq s_1$  and $q_0\neq q_1$,
the real interpolation spaces is still Besov spaces.
According to the layer language Lemmas \ref{lem:7.1} and   \ref{lem:lay} and equation \eqref{eq:basic2},
the translation of the above Theorems \ref{lem:layin} and \ref{lem:q01} is
\begin{corollary} \label{cor:7.10}
Given $0<\theta<1, 0<p, r, q_0, q_1\leq \infty$
and $s_0, s_1\in \mathbb{R}$. If $s_0\neq s_1$ and $s= (1-\theta) s_0 + \theta s_1$,
then for $f^{p}$ defined in the equations \eqref{eq:7.10000} and \eqref{eq:7.20000} by wavelet coefficients
and for $K(t,f, l^{s_0,q_0}, l^{s_1,q_1})$ defined in \eqref{eq:basic2},
we have the following wavelet characterization of
$K(t,f, B^{s_0,q_0}_{p}, B^{s_1,q_1}_{p})$:
$$K(t,f, B^{s_0,q_0}_{p}, B^{s_1,q_1}_{p})=  K(t,f^{p}, l^{s_0+\frac{n}{2}-\frac{n}{p},q_0}, l^{s_1+\frac{n}{2}-\frac{n}{p},q_1}). $$
Hence
$$(B^{s_0,q_0}_{p}, B^{s_1,q_1}_{p})_{\theta, r} = B^{s,r}_{p}.$$

\end{corollary}

\subsection{K functional for $s_0=s_1=s$ and $p_0=p_1=p$ } \label{sec:7.2}

When $s_0=s_1=s$ and $p_0=p_1=p$,
we know that the real interpolation spaces of Besov spaces are still Besov spaces
when $r=q$.
If $r\neq q$, the relative real interpolation spaces of Besov spaces run out the range Besov spaces
The known skills can not consider the cases where $r\neq q$.
We use a grid approach to handle such cases.
In fact, when $r\neq q$, it runs out of the Besov spaces and
becomes spaces with Lorentz indices to the frequency hierarchy.
Let $f^{s,p}$ be defined in Case 3 of Section \ref{sec:6.2} before Lemma \ref{lem:7.3}.
The interpolation space becomes the sequential Lorentz space of Sobolev space over the ring.
By Lemma \ref{lem:7.3}, the relative K functional is given by the following equation \eqref{eq:973}:
\begin{theorem}\label{th:7.11}
Given $s\in \mathbb{R}$ and $0<p, q_0\neq q_1\leq \infty$.
For $f^{s,p}$ defined in the equation \eqref{eq:7.30000} by wavelet coefficients and
$K(t, f^{s,p}, l^{q_0}, l^{q_1})$ obtained in Lemmas \ref{lem:5.1} and \ref{lem:5.2},
we have the following wavelet characterization of $K(t,f, B^{s,q_0}_{p}, B^{s,q_1}_{p})$:
\begin{equation}\label{eq:973}
K(t,f, B^{s,q_0}_{p}, B^{s,q_1}_{p}) = K(t, f^{s,p}, l^{q_0}, l^{q_1}).
\end{equation}
\end{theorem}

\begin{remark}
(i) When $r\neq q$, the real interpolation spaces in Theorem \ref{th:7.11} are not Besov-Lorentz spaces in \cite{BL} and \cite{Peetre},
they are new function spaces and have nonlinearity with respect to different frequencies.

(ii) The real interpolation spaces in Corollaries \ref{cor:7.7} and \ref{cor:7.10} are known as Besov spaces,
but we figured out the specific wavelet representation of their K functionals.

\end{remark}

\section{Compatibility and K functional at $q$ equal} \label{sec:8}
For $q_0=q_1=q$, the $K_q$ functional is compatible based on cuboid K functional
and the summation of the main grid is commutative with the minimal functional.
Lou-Yang-He-He \cite{LYHH}  have studied homogeneous Besov spaces
for the cases where $s_0=s_1=s$ and $q_0=q_1=q$,
which is first progress of Peetre's conjecture since 1976.
In this section, we extend Lou-Yang-He-He's Theorem to non-homogeneous cases
and generalize it to the case without restriction on $s_0$ and $s_1$.
For $j\geq 0$, write
\begin{equation}\label{eq:8.10000}
\tilde{f}_{j}= \{(Af)_{j,\gamma}\}_{\gamma\in \Gamma_j}= \{|f_{j,\gamma}|\}_{\gamma\in \Gamma_j}.
\end{equation}

\subsection{K functional for $0<p_0<p_1<\infty$ and $q=\infty$} \label{sec:881}
In this subsection, we study the commutativity of upper and lower bounds.
\begin{theorem}\label{th:8.1}
Given $s_0,s_1\in \mathbb{R}$ and $0<p_0<p_1<\infty$.
For $\tilde{f}_{j}$ defined in \eqref{eq:8.10000} by wavelet coefficients
and $K(t,f, l^{p_0}, l^{p_1}) $ obtained in Lemmas \ref{lem:6.4} and \ref{lem:6.5},
we have the following wavelet characterization of $K(t,f, B^{s_0,\infty}_{p_0}, B^{s_1,\infty}_{p_1})$:
\begin{equation} \label{eq:982}
K(t,f, B^{s_0,\infty}_{p_0}, B^{s_1,\infty}_{p_1})= \sup\limits_{j\geq 0} 2^{j(s_0 +\frac{n}{2}- \frac{n}{p_0})} K(t2^{j\tilde{s}},\tilde{f}_{j}, l^{p_0}, l^{p_1}).
\end{equation}
\end{theorem}

\begin{proof}
If $0<p_0<p_1<\infty$ and $q=\infty$, let
$$F(g_{j})= 2^{j(s_0 +\frac{n}{2}- \frac{n}{p_0})} (\sum\limits_{\gamma} g_{j,\gamma}^{p_0})^{\frac{1}{p_0}},\,\,
G(g_{j})= 2^{j(s_1 +\frac{n}{2}- \frac{n}{p_1})} (\sum\limits_{\gamma} |f_{j,\gamma}-g_{j,\gamma}|^{p_1})^{\frac{1}{p_1}}$$
We have
$$\sup\limits_{j\geq 0} (F(g_j) + t G(g_j))\leq \sup\limits_{j\geq 0} F(g_j) + t \sup\limits_{j\geq 0} G(g_j)
\leq 2\sup\limits_{j\geq 0} (F(g_j) + t G(g_j))$$

Hence,
$$\begin{array}{l}
K(t,f, B^{s_0,\infty}_{p_0}, B^{s_1,\infty}_{p_1})\\
= \inf\limits_{ g_{j,\gamma}\in \mathfrak{C}^{\Lambda}_{f} }
\{\sup\limits_{j\geq 0} 2^{j(s_0 +\frac{n}{2}- \frac{n}{p_0})} (\sum\limits_{\gamma} g_{j,\gamma}^{p_0})^{\frac{1}{p_0}}
+t \sup\limits_{j\geq 0} 2^{j(s_1 +\frac{n}{2}- \frac{n}{p_1})} (\sum\limits_{\gamma} |f_{j,\gamma}- g_{j,\gamma}|^{p_1})^{\frac{1}{p_1}}\}\\
\sim \inf\limits_{g_{j,\gamma}\in \mathfrak{C}^{\Lambda}_{f} }
\sup\limits_{j\geq 0} \{ 2^{j(s_0 +\frac{n}{2}- \frac{n}{p_0})} (\sum\limits_{\gamma} g_{j,\gamma}^{p_0})^{\frac{1}{p_0}}
+t 2^{j(s_1 +\frac{n}{2}- \frac{n}{p_1})} (\sum\limits_{\gamma} |f_{j,\gamma}- g_{j,\gamma}|^{p_1})^{\frac{1}{p_1}}\}\\
=\sup\limits_{j\geq 0} \inf\limits_{0\leq g_{j,\gamma}\leq |f_{j,\gamma}|}
\{ 2^{j(s_0 +\frac{n}{2}- \frac{n}{p_0})} (\sum\limits_{\gamma} g_{j,\gamma}^{p_0})^{\frac{1}{p_0}}
+t 2^{j(s_1 +\frac{n}{2}- \frac{n}{p_1})} (\sum\limits_{\gamma} |f_{j,\gamma}- g_{j,\gamma}|^{p_1})^{\frac{1}{p_1}}\}\\
=\sup\limits_{j\geq 0} K(t,\tilde{f}_{j}, \tilde{l}^{s_0,\infty}_{p_0}, \tilde{l}^{s_1,\infty}_{p_1})
\end{array}$$

Denote $\tilde{s}= s_1-s_0+ \frac{n}{p_0}-\frac{n}{p_1}$. By equation \eqref{eq:singlelayer},
$$\begin{array}{l}
K(t,f_{j}, \tilde{l}^{s_0,\infty}_{p_0}, \tilde{l}^{s_1,\infty}_{p_1})\\
\sim 2^{j(s_0 +\frac{n}{2}- \frac{n}{p_0})}
\inf\limits_{0\leq g_{j,\gamma}\leq |f_{j,\gamma}|}
\{  (\sum\limits_{\gamma} g_{j,\gamma}^{p_0})^{\frac{1}{p_0}}
+t 2^{j\tilde{s}} (\sum\limits_{\gamma} |f_{j,\gamma}- g_{j,\gamma}|^{p_1})^{\frac{1}{p_1}}\}\\
= 2^{j(s_0 +\frac{n}{2}- \frac{n}{p_0})} K(t2^{j\tilde{s}},\tilde{f}_{j}, l^{p_0}, l^{p_1})
\end{array}$$

\end{proof}

\subsection{K functional for $0<p_0<p_1<\infty$ and $0<q<\infty$} \label{sec:882}
In this subsection, we study the commutativity of $K_q$ functionals with respect to summability and lower bound.
By Theorem \ref{th:single}, $K(t,\tilde{f}_{j}, l^{s_0,q}_{p_0}, l^{s_0,q}_{p_1})$ is known and we have

\begin{theorem} \label{th:8.2}
Given $s_0,s_1\in \mathbb{R}, 0<p_0<p_1<\infty$ and $0<q<\infty$.
For $\tilde{f}_{j}$ defined in \eqref{eq:8.10000} by wavelet coefficients
and $K(t,\tilde{f}_{j}, B^{s_0,q}_{p_0}, B^{s_0,q}_{p_1})$ obtained by wavelet coefficients in Theorem \ref{th:single},
we have the following wavelet characterization of $K_q$ functional $K_q (t,f, B^{s_0,q}_{p_0}, B^{s_1,q}_{p_1})$:
\begin{equation} \label{eq:983}
\begin{array}{l}
K_q (t,f, B^{s_0,q}_{p_0}, B^{s_1,q}_{p_1})= \{ \sum\limits_{j\geq 0}  K(t,\tilde{f}_{j}, B^{s_0,q}_{p_0}, B^{s_0,q}_{p_1})^{q}\}^{\frac{1}{q}}.
\end{array}
\end{equation}
\end{theorem}

\begin{proof}
For $0<p_0<p_1<\infty$ and $0<q<\infty$, we have
$$\begin{array}{l}
\{K_{q}(t,f, B^{s_0,q}_{p_0}, B^{s_1,q}_{p_1})\}^{q}\\
= \inf\limits_{ g_{j,\gamma}\in \mathfrak{C}^{\Lambda}_{f}}
\{\sum\limits_{j\geq 0} 2^{qj(s_0 +\frac{n}{2}- \frac{n}{p_0})} (\sum\limits_{\gamma} g_{j,\gamma}^{p_0})^{\frac{q}{p_0}}
+t^{q} \sum\limits_{j\geq 0} 2^{qj(s_1 +\frac{n}{2}- \frac{n}{p_1})} (\sum\limits_{\gamma} |f_{j,\gamma}- g_{j,\gamma}|^{p_1})^{\frac{q}{p_1}}\}\\
\sim \inf\limits_{g_{j,\gamma}\in \mathfrak{C}^{\Lambda}_{f} }
\sum\limits_{j\geq 0} \{ 2^{qj(s_0 +\frac{n}{2}- \frac{n}{p_0})} (\sum\limits_{\gamma} g_{j,\gamma}^{p_0})^{\frac{q}{p_0}}
+t^{q} 2^{qj(s_1 +\frac{n}{2}- \frac{n}{p_1})} (\sum\limits_{\gamma} |f_{j,\gamma}- g_{j,\gamma}|^{p_1})^{\frac{q}{p_1}}\}\\
=\inf\limits_{ g_{j,\gamma}\in \mathfrak{C}^{\Lambda}_{f} } \sum\limits_{j\geq 0}
K^{q}_{q}(t, f_{j}, \tilde{l}^{s_0,q}_{p_0}, \tilde{l}^{s_1,q}_{p_1}, g_j)\\
\sim \sum\limits_{j\geq 0} \inf\limits_{0\leq g_{j,\gamma}\leq |f_{j,\gamma}|}
K^{q}_{q}(t, f_{j}, \tilde{l}^{s_0,q}_{p_0}, \tilde{l}^{s_1,q}_{p_1}, g_j) \\
=\sum\limits_{j\geq 0} 2^{qj(s_0 +\frac{n}{2}- \frac{n}{p_0})} K(t 2^{j\tilde{s}},\tilde{f}_{j}, l^{p_0}, l^{p_1})^{q}.
\end{array}$$

Hence
$$\begin{array}{l}
K_{q}(t,f, B^{s_0,q}_{p_0}, B^{s_1,q}_{p_1})\\
\sim  \{ \sum\limits_{j\geq 0} 2^{qj(s_0 +\frac{n}{2}- \frac{n}{p_0})} K(t 2^{j\tilde{s}},\tilde{f}_{j}, l^{p_0}, l^{p_1})^{q}\}^{\frac{1}{q}}\\
\sim  \{ \sum\limits_{j\geq 0}  K(t,\tilde{f}_{j}, B^{s_0,q}_{p_0}, B^{s_0,q}_{p_1})^{q}\}^{\frac{1}{q}}
\end{array}$$

\end{proof}

\subsection{K functional when $q_0=q_1$ equals and $p_1=\infty$} \label{sec:883}
For $q_0=q_1=q$ and $p_1=\infty$, we have similar conclusion.
For $\tilde{f}_{j}$ defined in \eqref{eq:8.10000} by wavelet coefficients
and $K(t,\tilde{f}_{j}, \tilde{l}^{s_0,q_0}_{p_0}, \tilde{l}^{s_1,q_1}_{p_1})$ obtained in Theorem \ref{th:6.60000},
we have the following wavelet characterization:

\begin{theorem}\label{th:8.3}
{\rm (i)} If $s_0,s_1\in \mathbb{R}, 0<p_0<p_1=\infty$ and $q=\infty$, then
$$\begin{array}{l}
K(t,f, B^{s_0,\infty}_{p_0}, B^{s_1,\infty}_{\infty})\\
\sim \sup\limits_{j\geq 0} 2^{j(s_0 +\frac{n}{2}- \frac{n}{p_0})}
\inf\limits_{0\leq g_{j,\gamma}\leq |f_{j,\gamma}|}
\{  (\sum\limits_{\gamma} g_{j,\gamma}^{p_0})^{\frac{1}{p_0}}
+t 2^{j\tilde{s}} (\sup\limits_{\gamma} |f_{j,\gamma}- g_{j,\gamma}| \}\\
= \sup\limits_{j\geq 0} 2^{j(s_0 +\frac{n}{2}- \frac{n}{p_0})} K(t2^{j\tilde{s}},\tilde{f}_{j}, l^{p_0}, l^{\infty}).
\end{array}$$

{\rm (ii)} If $s_0,s_1\in \mathbb{R}, 0<p_0<p_1=\infty$ and $0<q<\infty$, then
$$\begin{array}{l}
K_{q}(t,f, B^{s_0,q}_{p_0}, B^{s_1,q}_{\infty})\\
\sim  \{ \sum\limits_{j\geq 0} 2^{qj(s_0 +\frac{n}{2}- \frac{n}{p_0})} K(t 2^{j\tilde{s}},\tilde{f}_{j}, l^{p_0}, l^{p_1})^{q}\}^{\frac{1}{q}}\\
\sim  \{ \sum\limits_{j\geq 0}  K(t,\tilde{f}_{j}, \tilde{l}^{s_0,q}_{p_0}, \tilde{l}^{s_0,q}_{\infty})^{q}\}^{\frac{1}{q}}.
\end{array}$$
\end{theorem}

\section{$K_{\xi}^V$ functionals and vertex classification} \label{sec:9}

The idea of cuboid functional in sections 7 and 8 apply only to $p_0= p_1$ or $q_0=q_1$, respectively.
For more general cases, we need to do some functional calculus using the quasi-triangle inequality.
In fact, for general indices,
we can turn $K_{\xi}^{\mathfrak{C}}$ functionals defined on cuboid into vertex functionals $K_{\xi}^{V}$
defined on the vertex of the cuboid.

This provides a theoretical basis for us to give the concrete wavelet expression of the K functional of the complex index in the next step.

This prompted us to think about finding new ways to describe the nonlinear topology of complex situations.

Given $s_0,s_1\in \mathbb{R}, 0<p_0,p_1,q_0,q_1\leq \infty$.
Note that, by wavelet characterization,
\begin{equation*} \label{eq:kd}
\begin{array}{rcl} K _{\xi} (t,f, B^{s_0,q_0}_{p_0}, B^{s_1,q_1}_{p_1}, f_0) &=&
\lf\{ \lf[\sum\limits_{j\geq 0} 2^{j q_0 (s_0+ \frac{n}{2}-\frac{n}{p_0})}
(\sum\limits_{\gamma} (f_{j,\gamma}-f^{0}_{j,\gamma})^{p_0}) ^{\frac{q_0}{p_0}} \r] ^{\frac{\xi}{q_0}}\r.\\
&&\lf.+ t \lf[\sum\limits_{j\geq 0} 2^{j q_1 (s_1+ \frac{n}{2}-\frac{n}{p_1})}
(\sum\limits_{\gamma} (f^{0}_{j,\gamma}) ^{p_1}) ^{\frac{q_1}{p_1}} \r] ^{\frac{\xi}{q_1}}\r\}^{\frac{1}{\xi}}
\end{array}
\end{equation*}
Let $Af$ be the function  whose wavelet coefficients
are equal to the absolute value of the wavelet coefficients of $f$.
By Lemma \ref{eq:cuboid}, we have
$$K_{\xi}(t,f, \tilde{l}^{s_0,q_0}_{p_0}, \tilde{l}^{s_0,q_1}_{p_1})= K_{\xi}^{\mathfrak{C}}(t,Af, \tilde{l}^{s_0,q_0}_{p_0}, \tilde{l}^{s_1,q_1}_{p_1})$$

Let $ V^{\Lambda}_{f} = \{ g_{j,\gamma}=0 {\, \rm or \, } |f_{j,\gamma}|\}_{(j,\gamma)\in \Lambda}$ be the vertex point set of
the cuboid $ \mathfrak{C}^{\Lambda}_{f}= \{0\leq g_{j,\gamma}\leq |f_{j,\gamma}|\}_{(j,\gamma)\in \Lambda}$.
We convert the $K_{\xi}^{\mathfrak{C}}$ functional on the cuboid $ \mathfrak{C}^{\Lambda}_{f} = \{0\leq g_{j,\gamma}\leq |f_{j,\gamma}|\}_{(j,\gamma)\in \Lambda}$ to the vertex case on $ V^{\Lambda}_{f}$.
Denote
$$K_{\xi}^{V} (t,f, \tilde{l}^{s_0,q_0}_{p_0}, \tilde{l}^{s_1,q_1}_{p_1})
= \inf \limits_{f^{0}_{j,\gamma}\in V^{\Lambda}_{f} } K_{\xi} (t, Af, \tilde{l}^{s_0,q_0}_{p_0}, \tilde{l}^{s_1,q_1}_{p_1}, f_0).$$
If $\xi=1$, denote $K_{V}= K_{\xi}^{V}$.
The vertex functionals $K_{\xi}^{V}$ is equivalent to cuboid functionals $K_{\xi}^{\mathfrak{C}}$.
\begin{theorem}\label{th:vertex}
Given $s_0,s_1\in \mathbb{R}, 0<p_0,p_1,q_0,q_1\leq \infty$. Then\\
{\rm (i)} $K _{\xi}^{\mathfrak{C}} (t,f, \tilde{l}^{s_0,q_0}_{p_0}, \tilde{l}^{s_1,q_1}_{p_1}) \leq K_{\xi}^{V} (t,f, \tilde{l}^{s_0,q_0}_{p_0}, \tilde{l}^{s_1,q_1}_{p_1})$.\\
{\rm (ii)} $K_{\xi}^{V} (t,f, \tilde{l}^{s_0,q_0}_{p_0}, \tilde{l}^{s_1,q_1}_{p_1})\leq 2 K_{\xi}^{\mathfrak{C}} (t,f, \tilde{l}^{s_0,q_0}_{p_0}, \tilde{l}^{s_1,q_1}_{p_1})$.
\end{theorem}

\begin{proof}
(1) $K _{\xi}^{\mathfrak{C}} (t,f, \tilde{l}^{s_0,q_0}_{p_0}, \tilde{l}^{s_1,q_1}_{p_1})$ takes infimum for all possible $f_0$ in the cuboid,
but $K _{\xi}^{V} (t,f, \tilde{l}^{s_0,q_0}_{p_0}, \tilde{l}^{s_1,q_1}_{p_1})$ takes infimum for all possible $f_0$ in the vertex of the cuboid.
Hence $K _{\xi}^{\mathfrak{C}}(t,f, \tilde{l}^{s_0,q_0}_{p_0}, \tilde{l}^{s_1,q_1}_{p_1}) \leq K _{\xi}^{V} (t,f, \tilde{l}^{s_0,q_0}_{p_0}, \tilde{l}^{s_1,q_1}_{p_1})$.

(2) For all $Af_{j,\gamma}>0$, we must have $f^{0}_{j,\gamma}\leq \frac{1}{2} Af_{j,\gamma}$ or $f^{0}_{j,\gamma}> \frac{1}{2} Af_{j,\gamma}$.
For the first case, we denote $(j,\gamma)\in \Gamma_{1,j}$ and take $\tilde{f}^{0}_{j,\gamma}=0$.
For the second case, we denote $(j,\gamma)\in \Gamma_{2,j}$ and take $\tilde{f}^{0}_{j,\gamma}=Af_{j,\gamma}$.
We get $\tilde{f}^{0}$ and
$K_{\xi} (t,Af, B^{s_0,q_0}_{p_0}, B^{s_1,q_1}_{p_1}, \tilde{f}_0)=
\{ [\sum\limits_{j\geq 0} 2^{j q_0 (s_0+ \frac{n}{2}-\frac{n}{p_0})}
(\sum\limits_{\gamma\in \Gamma_{1,j}} |f_{j,\gamma}|^{p_0}) ^{\frac{q_0}{p_0}} ] ^{\frac{\xi}{q_0}}
+ t^{\xi} [\sum\limits_{j\geq 0} 2^{j q_1 (s_1+ \frac{n}{2}-\frac{n}{p_1})}
(\sum\limits_{\gamma \in \Gamma_{2,j}} (f^{0}_{j,\gamma}) ^{p_1}) ^{\frac{q_1}{p_1}} ] ^{\frac{\xi}{q_1}}\}^{\frac{1}{\xi}}.
$

Furthermore,
$$(\sum\limits_{\gamma\in \Gamma_{1,j}} |f_{j,\gamma}|^{p_0}) ^{\frac{1}{p_0}}
\leq 2(\sum\limits_{\gamma\in \Gamma_{1,j}} (|f_{j,\gamma}|-f^{0}_{j,\gamma})^{p_0}) ^{\frac{1}{p_0}},$$
$$(\sum\limits_{\gamma\in \Gamma_{2,j}} |f_{j,\gamma}|^{p_1}) ^{\frac{1}{p_1}}
\leq 2(\sum\limits_{\gamma\in \Gamma_{2,j}} (f^{0}_{j,\gamma})^{p_1}) ^{\frac{1}{p_1}}.$$
Hence
\begin{align*} &\{ K _{\xi} (t,Af, B^{s_0,q_0}_{p_0}, B^{s_1,q_1}_{p_1}, \tilde{f}_0)\}^{\xi} \\
\leq &2^{\xi} [\sum\limits_{j\geq 0} 2^{j q_0 (s_0+ \frac{n}{2}-\frac{n}{p_0})}
(\sum\limits_{\gamma\in \Gamma_{1,j}} (f_{j,\gamma}-f^{0}_{j,\gamma})^{p_0}) ^{\frac{q_0}{p_0}} ] ^{\frac{1}{q_0}}\\
&+ 2^{\xi} t^{\xi} [\sum\limits_{j\geq 0} 2^{j q_1 (s_1+ \frac{n}{2}-\frac{n}{p_1})}
(\sum\limits_{\gamma\in \Gamma_{2,j}} (f^{0}_{j,\gamma}) ^{p_1}) ^{\frac{q_1}{p_1}} ] ^{\frac{1}{q_1}}\\
\leq &2^{\xi} [\sum\limits_{j\geq 0} 2^{j q_0 (s_0+ \frac{n}{2}-\frac{n}{p_0})}
(\sum\limits_{\gamma} (f_{j,\gamma}-f^{0}_{j,\gamma})^{p_0}) ^{\frac{q_0}{p_0}} ] ^{\frac{1}{q_0}}\\
&+ 2 ^{\xi} t^{\xi} [\sum\limits_{j\geq 0} 2^{j q_1 (s_1+ \frac{n}{2}-\frac{n}{p_1})}
(\sum\limits_{\gamma} (f^{0}_{j,\gamma}) ^{p_1}) ^{\frac{q_1}{p_1}} ] ^{\frac{1}{q_1}}\\
\leq& 2 ^{\xi} \{ K_{\xi} (t,f, B^{s_0,q_0}_{p_0}, B^{s_1,q_1}_{p_1}, f_0)\}^{\xi}.
\end{align*}
\end{proof}

Theorem \ref{th:vertex} provides a possibility to consider K functional by vertexes classification.
We introduce the definition of the support set of wavelet coefficients
and the definition of Besov space supported on special discrete sets.

\begin{definition}
For $s\in \mathbb{R}, 0<p,q\leq \infty, S\subset \Lambda$ and $f(x)= \sum\limits_{(j,\gamma)\in \Lambda} f_{j,\gamma} \Phi_{j,\gamma}(x)$, we define

{\rm (i)} ${\rm Supp} f_{j,\gamma}= \{ (j,\gamma)\in \Lambda, f_{j,\gamma}\neq 0\}.$

{\rm (ii)} $B^{s,q}_{p}(S)= \{ \{f_{j,\gamma}\}_{(j,\gamma)\in \Lambda}\in \tilde{l}^{s,q}_{p} \mbox{\, and \,} {\rm Supp} f_{j,\gamma} \subset S\}.$

\end{definition}

The vertex K functional divides the support set of wavelet coefficients into two parts to be determined.
But the corresponding classification information is unclear.
We need to figure it out by analyzing the grid topology.
In the later of this paper, we use vertex functional to extend the rearrangement function of Lorentz \cite{Lorentz}
to the case where different given frequencies are compounded in a complex nonlinear manner.

\section{Vertexes functional and $p_0\neq p_1$ and $q_0=q_1$ }\label{sec:9b}

In this section we use vertex K functionals to give a short proof of the equivalent expression of K functionals when $q_0=q_1$.
For $j\geq 0$, for $f_{j}$ defined in \eqref{eq:6.30000} by wavelet coefficients and for $\Gamma^{0}_{j}(t)$ and $\Gamma^{1}_{j}(t)$ defined in Section \ref{sec:5.2}, denote
\begin{equation}\label{eq:10.10000}
A_{j}= \|f_j\|_{B^{s_0,q_0}_{p_0}(\Gamma^{0}_{j}(t2^{j\tilde{s}}))} \mbox{ and } B_{j}= \|f_j\|_{B^{s_1,q_1}_{p_1}(\Gamma^{1}_{j}(t2^{j\tilde{s}}))}.
\end{equation}
By Lemmas \ref{lem:6.4}, \ref{lem:6.5}  and Theorem \ref{th:single}, we have
\begin{equation}\label{eq:ABC}
A_{j} + t B_{j}=  K(t, f_{j}, B^{s_0,q_0}_{p_0}, B^{s_1,q_1}_{p_1}).
\end{equation}
Denote
\begin{equation}
\label{eq:9b2}
f^{0}_{j}= f_j |_{\Gamma^{0}_{j}(t2^{j\tilde{s}})} \mbox{\, and \,} f^{1}_{j}= f_j|_{\Gamma^{1}_{j}(t2^{j\tilde{s}})}.
\end{equation}

For $q_0=q_1=q$ and $p_0<p_1$, we have the following wavelet characterization:
\begin{theorem}\label{th:10.1}
Given $s_0,s_1\in \mathbb{R}, 0<p_0<p_1\leq  \infty$  and $0<q\leq \infty$.
By equations \eqref{eq:10.10000} and \eqref{eq:ABC},
$A_{j} $, $B_{j}$ and $K(t, f_{j}, B^{s_0,q_0}_{p_0}, B^{s_1,q_1}_{p_1})$
are all defined by wavelet coefficients, we have
the following wavelet characterization of $K (t,f, B^{s_0,q}_{p_0}, B^{s_1,q}_{p_1})$:
\begin{equation}
\begin{array}{rcl}
K (t,f, B^{s_0,q}_{p_0}, B^{s_1,q}_{p_1})
&\sim &  \{\sum\limits_{j\geq 0} A_j^{q}\}^{\frac{1}{q}} + t \{\sum\limits_{j\geq 0} B_j^{q}\}^{\frac{1}{q}}\\
&\sim &
\{\sum\limits_{j\geq 0} K(t, f_{j}, B^{s_0,q_0}_{p_0}, B^{s_1,q_1}_{p_1})^{q}\}^{\frac{1}{q}} .
\end{array}
\end{equation}
\end{theorem}

\begin{proof}

Let's call $S^{0}$ and $S^{1}$ a partition of $\Lambda$, if
$$S^{0}\bigcap S^{1}=\phi \mbox{ \, and \, } S^{0}\bigcup S^{1}= \Lambda.$$
By H\"older inequality, we know
$$\begin{array}{rl}
&\|f\|_{B^{s_0,q}_{p_0}(S^{0})} + t \|f\|_{B^{s_1,q}_{p_1}(S^{1})}\\
\sim &
\{\|f\|_{B^{s_0,q}_{p_0}(S^{0})}^{q} + t^{q} \|f\|_{B^{s_1,q}_{p_1}(S^{1})}^{q}\}^{\frac{1}{q}}\\
\sim &
\{ \sum\limits_{j\in \mathbb{N}} [ \, \|f_{j}\|_{B^{s_0,q}_{p_0}(S^{0})}^{q}  + t ^{q} \|f_{j}\|_{B^{s_1,q}_{p_1}(S^{1})} ^{q}\, ]
\}^{\frac{1}{q}}\\
\sim &
\{ \sum\limits_{j\in \mathbb{N}} [ \, \|f_{j}\|_{B^{s_0,q}_{p_0}(S^{0})} + t \|f_{j}\|_{B^{s_1,q}_{p_1}(S^{1})}\, ] ^{q}
\}^{\frac{1}{q}}.
\end{array}$$

Applying the definition of the vertex K functional and the vertex K functional Theorem \ref{th:vertex},
taking the limit of the above equations,
we get the conclusion of the Theorem \ref{th:10.1}.

\end{proof}

By Lemma \ref{lem:CK} and Theorem \ref{th:10.1}, we have
\begin{corollary}\label{cor:10.1}
Given $s_0,s_1\in \mathbb{R}, 0<p_1<p_0\leq  \infty$ and $0<q\leq \infty$.
By equations \eqref{eq:10.10000} and \eqref{eq:ABC},
$K(t, f_{j}, B^{s_1,q_1}_{p_1}, B^{s_0,q_0}_{p_0})$
is defined by wavelet coefficients, we have
the following wavelet characterization of $K (t,f, B^{s_0,q}_{p_0}, B^{s_1,q}_{p_1})$:
\begin{equation}
\begin{array}{rcl}
K (t,f, B^{s_0,q}_{p_0}, B^{s_1,q}_{p_1})
&\sim &  t K (t^{-1},f, B^{s_1,q}_{p_1}, B^{s_0,q}_{p_0})\\
&\sim &
t \{\sum\limits_{j\geq 0} K(t^{-1}, f_{j}, B^{s_1,q}_{p_1}, B^{s_0,q_0}_{p_0})^{q}\}^{\frac{1}{q}} .
\end{array}
\end{equation}
\end{corollary}

\begin{remark}
(i) Theorem \ref{th:10.1} and Corollary \ref{cor:10.1} provide a simple wavelet characterization of K functional for $q_0=q_1$
which generalized Lou-Yang-He-He's result \cite{LYHH} to the case where $s_0\neq s_1$.
Further, by applying vertex K functional, the proof of Theorem \ref{th:10.1} is much easier.

(ii) By Lemma \ref{lem:Besov-Lorentz} and Theorem \ref{th:10.1}, it is easy to see that
the relative real interpolation spaces in the main theorem \ref{th:main} are Besov-Lorentz spaces
when $q_0=q_1=q=r$.
{\bf For $r\neq q$, the relative real interpolation spaces are new function spaces}.
\end{remark}

\section{Three nonlinear topology structure and K functional}\label{sec:10}
We consider first some relation of $K_{\infty}$ functional for psedo-norm spaces and
their exponential spaces.
Then we give the specific wavelet expression of the K functional
for $s_0, s_1\in \mathbb{R}, 0<p_0\neq p_1\leq \infty, 0< q_0\neq q_1<\infty$.
In fact, we introduce first power spaces on the full grid.
That is to say,
we introduce the power space of Besov space.
$$\|f\|_{X_0} = \|f\|^{q_0}_{\tilde{l}^{s_0,q_0}_{p_0}(\Lambda)} \mbox{ and }
\|f\|_{X_1}= \|f\|^{q_1}_{\tilde{l}^{s_1, q_1}_{p_1}(\Lambda)}.$$
We introduce then power spaces on the layer grid.
For $j\geq 0$, denote
$$\|f\|_{Y^j_0} = \|f\|^{q_0}_{l^{p_0}(\Gamma_j)} \mbox { and }
\|f\|_{Y^j_1}= \|f\|^{q_1}_{l^{p_1}(\Gamma_j)}.$$
{\bf We establish the topology change on the layer grid and the topology change on the full grid respectively. And establish the topological compatibility of K functional: summation of main grid and minimal functional commutativity. }
By introducing power spaces,
we use three Lemmas \ref{cor:11.1}, \ref{lem:ab} and \ref{cor:11.2} to describe the corresponding nonlinearity one by one.
Finally, the corresponding wavelet characterization is given in Theorem \ref{th:11.4}.

We introduce first $K_{\infty}$ functional.
We denote $\Lambda_0 \bigoplus \Lambda_1= \Lambda$, if
(i) $\Lambda_0 \bigcup \Lambda_1= \Lambda$ and $\Lambda_0  \bigcap \Lambda_1= \phi$. Define
\begin{equation} \label{eq:11.1}
K_{\infty}(t,f, A_0, A_1)= \inf\limits_{\Lambda_0 \bigoplus \Lambda_1= \Lambda} \max (\|f\|_{A_0(\Lambda_0)} ,  t\|f\|_{A_1(\Lambda_1)}).
\end{equation}
It is easy to see that $K_{\infty}$ functional in the above \eqref{eq:11.1} has the same properties as $K_{\xi}$ functional
where $0<\xi<\infty$. We omit the detail and omit the proof.

The first level of nonlinearity is that of the K functional at a particular frequency level.
We consider nonlinear topological deformation properties of layer grid functional
and get $K_{\infty}(t,b_{j}, Y^j_0, Y^j_1)$.
Given $j\geq 0, 0<p_0\neq p_1\leq \infty$.
We describe the first type of nonlinearity in terms of the magnitude of $q_0$ and $q_1$ in two cases.
Denote $b_{j}= (b_{j,\gamma})_{\gamma\in \Gamma_{j}}$
where $b_{j,\gamma}= 2^{j(s_0+\frac{n}{2}-\frac{n}{p_0})} |f_{j,\gamma}|$.

{\bf Case (i).}
If $0<q_1<q_0< \infty$,  denote
$$s= t^{q_1} K_{\infty} (t,b_{j}, l^{p_0}(\Gamma_j), l^{p_1}(\Gamma_j))^{q_0 -q_1} .$$
For fixed $b_{j}$, by monotonicity, there exists a function $F_{j}$ that describes
the nonlinearity of $t$ as a function of $s$  such that
\begin{equation}\label{eq:t.3}
t= F_{j}(b_{j},l^{p_0}(\Gamma_j), l^{p_1}(\Gamma_j), q_0, q_1, s).
\end{equation}

{\bf Case (ii).}
If $0<q_0<q_1< \infty$,  denote
$$s^{-1}= t^{q_0} K_{\infty} (t,b_{j}, l^{p_1}(\Gamma_j), l^{p_0}(\Gamma_j))^{q_1 -q_0} .$$
Similar to Case (i), by monotonicity, there exists $F_{j}$ such that
\begin{equation}\label{eq:t.4}
t= F_{j}(b_{j},l^{p_1}(\Gamma_j), l^{p_0}(\Gamma_j), q_1, q_0, s^{-1}).
\end{equation}

The first level of nonlinearity is for fixed frequencies.
It refers to the representation of the K functional of the power space of fixed frequency
in a nonlinear way as the K functional on the grid $\Gamma_{j}$.
We have
\begin{lemma} \label{cor:11.2}
Given $j\geq 0$ and $ 0<p_0\neq p_1\leq \infty$. For $0<s,t<\infty$, we have

{\rm (i)} For $0<q_1<q_0< \infty$ and
$t= F_{j}(b_{j},l^{p_0}(\Gamma_j), l^{p_1}(\Gamma_j), q_0, q_1, s)$ defined in the equation \eqref{eq:t.3},
we have
\begin{equation}\label{eq:K11.3}K_{\infty}(s,b_{j}, Y^j_0, Y^j_1)= K_{\infty}(t,b_{j}, l^{p_0}(\Gamma_j), l^{p_1}(\Gamma_j))^{q_0}.
\end{equation}

{\rm (ii)} For $0<q_0<q_1< \infty$
and $t= F_{j}(b_{j},l^{p_1}(\Gamma_j), l^{p_0}(\Gamma_j), q_1, q_0, s^{-1})$
defined in the equation \eqref{eq:t.4},
we have
\begin{equation}\label{eq:K11.4}
K_{\infty}(s,b_{j}, Y^j_0, Y^j_1)= s K_{\infty}(s^{-1},b_{j}, Y^j_1, Y^j_0)
= s K_{\infty}(t,b_{j}, l^{p_1}(\Gamma_j), l^{p_0}(\Gamma_j))^{q_1}.
\end{equation}

\end{lemma}

\begin{proof}
(i) For $0<q_1<q_0< \infty$, by applying \eqref{eq:t.3},
$$\begin{array}{l} K_{\infty}(s,b_{j}, Y^j_0, Y^j_1)= \inf\limits_{\Gamma^{0}_j \bigoplus \Gamma^{1}_j= \Gamma_{j}}
\max (\|b_{j}\|_{Y^j_0}, s\|b_{j}\|_{Y^j_1})\\
= \inf\limits_{\Gamma^{0}_j \bigoplus \Gamma^{1}_j= \Gamma_{j}}
\max (\|b_{j}\|^{q_0}_{l^{p_0}}, s\|b_{j}\|^{q_1}_{l^{p_1}})
\\
=K_{\infty}(t,b_{j}, l^{p_0}(\Gamma_j), l^{p_1}(\Gamma_j))^{q_0}
\inf\limits_{\Gamma^{0}_j \bigoplus \Gamma^{1}_j= \Gamma_{j}}
\max ( (\frac{\|b_{j}\|_{l^{p_0}(\Lambda_0)}} {K_{\infty}(t,b_{j}, l^{p_0}(\Gamma_j), l^{p_1}(\Gamma_j))})^{q_0},
(\frac{t \|b_{j}\|_{l^{p_1}(\Lambda_1)}} {K_{\infty}(t,b_{j}, l^{p_0}(\Gamma_j), l^{p_1}(\Gamma_j)))})^{q_1})
.\end{array}$$
Since
$$1=\inf\limits_{\Gamma^{0}_j \bigoplus \Gamma^{1}_j= \Gamma_{j}}
\max ( (\frac{\|b_{j}\|_{l^{p_0}(\Lambda_0)}} {K_{\infty}(t,b_{j}, l^{p_0}(\Gamma_j), l^{p_1}(\Gamma_j))})^{q_0},
(\frac{t \|b_{j}\|_{l^{p_1}(\Lambda_1)}} {K_{\infty}(t,b_{j}, l^{p_0}(\Gamma_j), l^{p_1}(\Gamma_j)))})^{q_1}),$$
we get \eqref{eq:K11.3}.

(ii) For $0<q_0<q_1< \infty$, by applying \eqref{eq:t.4},
$$\begin{array}{l} K_{\infty}(s^{-1},b_{j}, Y^j_1, Y^j_0)= \inf\limits_{\Gamma^{0}_j \bigoplus \Gamma^{1}_j= \Gamma_{j}}
\max (\|b_{j}\|_{Y^j_1}, s^{-1} \|b_{j}\|_{Y^j_0})\\
= \inf\limits_{\Gamma^{0}_j \bigoplus \Gamma^{1}_j= \Gamma_{j}}
\max (\|b_{j}\|^{q_1}_{l^{p_1}}, s^{-1}\|b_{j}\|^{q_0}_{l^{p_0}})
\\
=K_{\infty}(t,b_{j}, l^{p_1}(\Gamma_j), l^{p_0}(\Gamma_j))^{q_1}
\inf\limits_{\Gamma^{0}_j \bigoplus \Gamma^{1}_j= \Gamma_{j}}
\max ( (\frac{\|b_{j}\|_{l^{p_1}(\Lambda_0)}} {K_{\infty}(t,b_{j}, l^{p_1}(\Gamma_j), l^{p_0}(\Gamma_j))})^{q_1},
(\frac{t \|b_{j}\|_{l^{p_0}(\Lambda_1)}} {K_{\infty}(t,b_{j}, l^{p_1}(\Gamma_j), l^{p_0}(\Gamma_j)))})^{q_0})
.\end{array}$$
Since
$$1=\inf\limits_{\Gamma^{0}_j \bigoplus \Gamma^{1}_j= \Gamma_{j}}
\max ( (\frac{\|b_{j}\|_{l^{p_1}(\Lambda_1)}} {K_{\infty}(t,b_{j}, l^{p_1}(\Gamma_j), l^{p_0}(\Gamma_j))})^{q_1},
(\frac{t \|b_{j}\|_{l^{p_0}(\Lambda_0)}} {K_{\infty}(t,b_{j}, l^{p_1}(\Gamma_j), l^{p_0}(\Gamma_j)))})^{q_0}),$$
we get \eqref{eq:K11.4}.

\end{proof}




We consider then the topological compatibility of intermediate spaces
and obtain $K_{\infty}(t,f, X_0, X_1)$.
The second level of nonlinearity is the commutativity on the sum relative to indices $j$ and infimum in the K functional of Besov power space.
The K-functionals of the power space of the Besov space can be given by a set of K-functionals of the power spaces defined on layer grid.
\begin{lemma} \label{lem:ab}
For $s_0, s_1\in \mathbb{R}$, $0<p_0\neq p_1\leq \infty, 0< q_0\neq q_1<\infty$,
we have
\begin{equation}\label{eq:11.60000}
\begin{array}{l}
K_{\infty}(t,f, X_0, X_1)
\sim \sum\limits_{j\in \mathbb{N}} K_{\infty}(t2^{j\tilde{s}}, b_{j}, Y^j_0, Y^j_1).
\end{array}
\end{equation}
\begin{equation}\label{eq:11.70000}
\begin{array}{l}
K_{\infty}(t,f, X_1, X_0) = t K_{\infty} (t^{-1}, f, X_0, X_1)
\sim t \sum\limits_{j\in \mathbb{N}} K_{\infty}(t^{-1} 2^{j\tilde{s}}, b_{j}, Y^j_0, Y^j_1).
\end{array}
\end{equation}
\end{lemma}

\begin{proof}
The proof of the equation \eqref{eq:11.70000} is similar to which of the equation \eqref{eq:11.60000}.
We prove only the equation \eqref{eq:11.60000}.
For $i=0,1$, let $\Lambda_i$ be defined before equation \eqref{eq:11.1}.
For $i=0,1, j\geq 0$, denote $\Gamma^{i}_{j}= \{\gamma\in \Gamma_{j}, (j,\gamma)\in \Lambda_i\}$.
We have
$$\begin{array}{l}
K_{\infty}(t,f, X_0, X_1)\\
= \inf\limits_{\Lambda_0 \bigoplus \Lambda_1= \Lambda} \max (
\sum\limits_{j\in \mathbb{N}} 2^{j(s_0+\frac{n}{2}-\frac{n}{p_0}) q_0} (\sum\limits_{\gamma\in \Gamma^{0}_{j}} |f_{j,\gamma}|^{p_0} )^{\frac{q_0}{p_0}},
t\sum\limits_{j\in \mathbb{N}} 2^{j(s_1+\frac{n}{2}-\frac{n}{p_1}) q_1} (\sum\limits_{\gamma\in \Gamma^{1}_{j}} |f_{j,\gamma}|^{p_1} )^{\frac{q_1}{p_1}})\\
\sim \inf\limits_{\Lambda_0 \bigoplus \Lambda_1= \Lambda}
\sum\limits_{j\in \mathbb{N}} 2^{j(s_0+\frac{n}{2}-\frac{n}{p_0}) q_0} (\sum\limits_{\gamma\in \Gamma^{0}_{j}} |f_{j,\gamma}|^{p_0} )^{\frac{q_0}{p_0}}
+t\sum\limits_{j\in \mathbb{N}} 2^{j(s_1+\frac{n}{2}-\frac{n}{p_1}) q_1} (\sum\limits_{\gamma\in \Gamma^{1}_{j}} |f_{j,\gamma}|^{p_1} )^{\frac{q_1}{p_1}}\\
\sim
\sum\limits_{j\in \mathbb{N}} \inf\limits_{\Gamma^{0}_j \bigoplus \Gamma^{1}_j= \Gamma_{j} }
\{2^{j(s_0+\frac{n}{2}-\frac{n}{p_0}) q_0} (\sum\limits_{\gamma\in \Gamma^{0}_{j}} |f_{j,\gamma}|^{p_0} )^{\frac{q_0}{p_0}}
+ t 2^{j(s_1+\frac{n}{2}-\frac{n}{p_1}) q_1} (\sum\limits_{\gamma\in \Gamma^{1}_{j}} |f_{j,\gamma}|^{p_1} )^{\frac{q_1}{p_1}}\}\\
\sim
\sum\limits_{j\in \mathbb{N}} \inf\limits_{\Gamma^{0}_j \bigoplus \Gamma^{1}_j= \Gamma_{j} }
\max \{2^{j(s_0+\frac{n}{2}-\frac{n}{p_0}) q_0} (\sum\limits_{\gamma\in \Gamma^{0}_{j}} |f_{j,\gamma}|^{p_0} )^{\frac{q_0}{p_0}},
t 2^{j(s_1+\frac{n}{2}-\frac{n}{p_1}) q_1} (\sum\limits_{\gamma\in \Gamma^{1}_{j}} |f_{j,\gamma}|^{p_1} )^{\frac{q_1}{p_1}}\}\\
\sim
\sum\limits_{j\in \mathbb{N}} \inf\limits_{\Gamma^{0}_j \bigoplus \Gamma^{1}_j= \Gamma_{j} }
\max \{(\sum\limits_{\gamma\in \Gamma^{0}_{j}} |b_{j,\gamma}|^{p_0} )^{\frac{q_0}{p_0}},
t 2^{j\tilde{s}} (\sum\limits_{\gamma\in \Gamma^{1}_{j}} |b_{j,\gamma}|^{p_1} )^{\frac{q_1}{p_1}}\}
\end{array}
$$
Hence we get equation \eqref{eq:11.60000}.

\end{proof}



We describe the third type of nonlinearity in terms of the magnitude of $q_0$ and $q_1$ in two cases.
Thirdly, we consider topological deformation properties of full grid
and get $K_{\infty}(t,f, \tilde{l}^{s_0,q_0}_{p_0}(\Lambda), \tilde{l}^{s_1, q_1}_{p_1}(\Lambda))$.

{\bf Case (i).} For $0< q_0<  q_1< \infty$,
denote
$$s= t^{\frac{1}{q_1}} K_{\infty} (t,f, X_0, X_1)^{\frac{1}{q_0} -\frac{1}{q_1}} .$$
Given $X_0$ and $X_1$.
For fixed function $f$, by monotonicity, there exists a function $F$ that describes
the nonlinearity of $t$ as a function of $s$  such that
\begin{equation} \label{eq:t.1}
t= F(f,X_0,X_1, \frac{1}{q_0}, \frac{1}{q_1}, s).
\end{equation}

{\bf Case (ii).} For $0< q_1<  q_0< \infty$,
denote
$$s^{-1}= t^{\frac{1}{q_0}} K_{\infty} (t,f, X_1, X_0)^{\frac{1}{q_1} -\frac{1}{q_0}} .$$
Similar to the equation \eqref{eq:t.1}, by monotonicity, there exists function $F$ such that
\begin{equation} \label{eq:t.2}
t= F(f,X_1,X_0, \frac{1}{q_1}, \frac{1}{q_0}, s^{-1}).
\end{equation}

Hence the K functional of Besov spaces is transformed into the K functional of Besov power spaces:
\begin{lemma} \label{cor:11.1}
Given $s_0, s_1\in \mathbb{R}$ and $0<p_0\neq p_1\leq \infty$. For $0<s,t<\infty$, we have

{\rm (i)} If $0<q_0<  q_1< \infty$ and $t$ is defined in \eqref{eq:t.1}, then
\begin{equation}\label{eq:K11.1}
K_{\infty}(s,f, \tilde{l}^{s_0,q_0}_{p_0}(\Lambda), \tilde{l}^{s_1, q_1}_{p_1}(\Lambda))= K_{\infty}(t,f, X_0, X_1)^{\frac{1}{q_0}}.
\end{equation}

{\rm (ii)} If $0<q_1<  q_0< \infty$  and $t$ is defined in \eqref{eq:t.2}, then
\begin{equation}\label{eq:K11.2}
\begin{array}{rcl}
K_{\infty}(s,f, \tilde{l}^{s_0,q_0}_{p_0}(\Lambda), \tilde{l}^{s_1, q_1}_{p_1}(\Lambda))
&=& s K_{\infty}(s^{-1},f, \tilde{l}^{s_1, q_1}_{p_1}(\Lambda), \tilde{l}^{s_0,q_0}_{p_0}(\Lambda))\\
&=& s K_{\infty}(t,f, X_1, X_0)^{\frac{1}{q_1}}.
\end{array}
\end{equation}

\end{lemma}

\begin{proof}
(i) For $0<q_0<  q_1< \infty$, by applying \eqref{eq:t.1},
$$\begin{array}{l}K_{\infty}(s,f, \tilde{l}^{s_0,q_0}_{p_0}(\Lambda), \tilde{l}^{s_1, q_1}_{p_1}(\Lambda)) =
\inf\limits_{\Lambda_0 \bigoplus \Lambda_1= \Lambda} \max (\|f\|_{\tilde{l}^{s_0,q_0}_{p_0}(\Lambda_0)} ,  s\|f\|_{\tilde{l}^{s_1, q_1}_{p_1}(\Lambda_1)})\\
= \inf\limits_{\Lambda_0 \bigoplus \Lambda_1= \Lambda} \max (\|f\|^{\frac{1}{q_0}}_{X_{0}(\Lambda_0)} ,
s\|f\|^{\frac{1}{q_1}}_{X_{1}(\Lambda_1)})\\
= K_{\infty}(t,f, X_0, X_1)^{\frac{1}{q_0}} \inf\limits_{\Lambda_0 \bigoplus \Lambda_1= \Lambda} \max (
(\frac{ \|f\| _{X_{0}(\Lambda_0)} } { K_{\infty} (t,f, X_0, X_1)}) ^{ \frac{1}{q_0} } ,
(\frac{t\|f\|_{X_{1}(\Lambda_1)}} { K_{\infty} (t,f, X_0, X_1)} )^{\frac{1}{q_1}} ).
\end{array}$$

Since
$$1=\inf\limits_{\Lambda_0 \bigoplus \Lambda_1= \Lambda}
\max ( (\frac{\|f\|_{X_0(\Lambda_0)}} {K_{\infty}(t,f, X_0, X_1)})^{\frac{1}{q_0}}, (\frac{t \|f\|_{X_1(\Lambda_1)}} {K_{\infty}(t,a, X_0, X_1))})^{\frac{1}{q_1}}),$$
we have \eqref{eq:K11.1}.

(ii) For $0<q_1<  q_0< \infty$, by applying \eqref{eq:t.2},
$$\begin{array}{l}K_{\infty}(s,f, \tilde{l}^{s_0,q_0}_{p_0}(\Lambda), \tilde{l}^{s_1, q_1}_{p_1}(\Lambda)) =
s K_{\infty}(s^{-1},f, \tilde{l}^{s_1, q_1}_{p_1}(\Lambda), \tilde{l}^{s_0,q_0}_{p_0}(\Lambda))\\
= s \inf\limits_{\Lambda_0 \bigoplus \Lambda_1= \Lambda} \max ( \|f\|_{\tilde{l}^{s_1, q_1}_{p_1}(\Lambda_1)},
s^{-1} \|f\|_{\tilde{l}^{s_0,q_0}_{p_0}(\Lambda_0)} )\\
= s \inf\limits_{\Lambda_0 \bigoplus \Lambda_1= \Lambda} \max (
\|f\|^{\frac{1}{q_1}}_{X_{1}(\Lambda_1)},
s^{-1}\|f\|^{\frac{1}{q_0}}_{X_{0}(\Lambda_0)}
)\\
= s K_{\infty}(t,f, X_1, X_0)^{\frac{1}{q_1}} \inf\limits_{\Lambda_0 \bigoplus \Lambda_1= \Lambda} \max (
(\frac{ \|f\| _{X_{1}(\Lambda_0)} } { K_{\infty} (t,f, X_1, X_0)}) ^{ \frac{1}{q_1} } ,
(\frac{t\|f\|_{X_{0}(\Lambda_1)}} { K_{\infty} (t,f, X_1, X_0)} )^{\frac{1}{q_0}} ).
\end{array}$$

Since
$$1=\inf\limits_{\Lambda_0 \bigoplus \Lambda_1= \Lambda}
\max ( (\frac{\|f\|_{X_1(\Lambda_0)}} {K_{\infty}(t,f, X_1, X_0)})^{\frac{1}{q_1}}, (\frac{t \|f\|_{X_0(\Lambda_1)}} {K_{\infty}(t,f, X_1, X_0))})^{\frac{1}{q_0}}),$$
we have \eqref{eq:K11.2}.
\end{proof}


Combining the above three levels of nonlinearity,
by applying Lemmas \ref{lem:CK}, \ref{cor:11.1}, \ref{lem:ab} and \ref{cor:11.2},
we obtain the following wavelet characterization of the K functional:

\begin{theorem} \label{th:11.4}

Given $s_0,s_1\in \mathbb{R}$, $0<p_0\neq p_1\leq \infty, 0<q_0\neq  q_1< \infty$ and $0<t<\infty$.
For $(j,\gamma)\in \Lambda$, let $b_{j,\gamma}= 2^{j(s_0+\frac{n}{2}-\frac{n}{p_0})} |f_{j,\gamma}|$
and $b_{j}= (b_{j,\gamma})_{\gamma\in \Gamma_{j}}$.

{\rm (i)} For $q_0<q_1$, let $s= F(f,X_0,X_1, \frac{1}{q_0}, \frac{1}{q_1}, t)$ be defined in \eqref{eq:t.1}
and let $\tau= F_{j}(b_{j},l^{p_1}(\Gamma_j), l^{p_0}(\Gamma_j), q_1, q_0, s^{-1}2^{-j\tilde{s}})$ be defined in \eqref{eq:t.4}.
Let $K_{\infty}(\tau,b_{j}, l^{p_0}(\Gamma_j), l^{p_1}(\Gamma_j))$ be defined by wavelet coefficients
in Subsection \ref{sec:5.2}.
The wavelet characterization of $K_{\infty}(s,f, B^{s_0,q_0}_{p_0}, B^{s_1, q_1}_{p_1})$ is defined as follows:
$$\begin{array}{rcl}
K_{\infty}(t,f, B^{s_0,q_0}_{p_0}, B^{s_1, q_1}_{p_1})&=&
\{\sum\limits_{j\in \mathbb{N}} s2^{j\tilde{s}} K_{\infty} (\tau, b_{j}, l^{p_1}(\Gamma_j), l^{p_0}(\Gamma_j))^{q_1} \}^{\frac{1}{q_0}}.
\end{array}$$

{\rm (ii)} For $q_0>q_1$, let $s = F(f,X_1,X_0, \frac{1}{q_1}, \frac{1}{q_0}, t^{-1})$ be defined in \eqref{eq:t.2} and
let $\tau=F_{j}(b_{j},l^{p_0}(\Gamma_j), l^{p_1}(\Gamma_j), q_0, q_1, s^{-1} 2^{j\tilde{s}})$ be defined in \eqref{eq:t.3}.
Let $K_{\infty}(\tau,b_{j}, l^{p_0}(\Gamma_j), l^{p_1}(\Gamma_j))$ be defined by wavelet coefficients in Subsection \ref{sec:5.2}.
The wavelet characterization of $K_{\infty}(s,f, B^{s_0,q_0}_{p_0}, B^{s_1, q_1}_{p_1})$ is defined as follows:

$$\begin{array}{rcl}
K_{\infty}(t,f, B^{s_0,q_0}_{p_0}, B^{s_1, q_1}_{p_1})
&=& t\cdot s^{\frac{1}{q_1}} \{\sum\limits_{j\in \mathbb{N}}
[ K_{\infty}(\tau,b_{j}, l^{p_0}(\Gamma_j), l^{p_1}(\Gamma_j))]^{q_0}
\}^{\frac{1}{q_1}}.
\end{array}$$
\end{theorem}


\begin{proof}
{\rm (i)} If $q_0<q_1$, then
$$\begin{array}{rcl}
K_{\infty}(t,f, B^{s_0,q_0}_{p_0}, B^{s_1, q_1}_{p_1})&=& K_{\infty}(t,f, \tilde{l}^{s_0,q_0}_{p_0}(\Lambda), \tilde{l}^{s_1, q_1}_{p_1}(\Lambda))\\
&=& \{\sum\limits_{j\in \mathbb{N}} K_{\infty} (s2^{j\tilde{s}},b_{j}, Y^j_0, Y^j_1) \}^{\frac{1}{q_0}}\\
&=& \{\sum\limits_{j\in \mathbb{N}} s2^{j\tilde{s}} K_{\infty} (s^{-1}2^{-j\tilde{s}},b_{j}, Y^j_1, Y^j_0) \}^{\frac{1}{q_0}}.
\end{array}$$
Further
$$\begin{array}{rcl}
 K_{\infty} (s^{-1}2^{-j\tilde{s}},b_{j}, Y^j_1, Y^j_0)
&=&  K_{\infty} (\tau, b_{j}, l^{p_1}(\Gamma_j), l^{p_0}(\Gamma_j))^{q_1}.
\end{array}$$
Hence we get (i).

{\rm (ii)} If $q_0>q_1$, then
$$\begin{array}{rcl}
K_{\infty}(t,f, B^{s_0,q_0}_{p_0}, B^{s_1, q_1}_{p_1})&=&
K_{\infty}(t,f, \tilde{l}^{s_0,q_0}_{p_0}(\Lambda), \tilde{l}^{s_1, q_1}_{p_1}(\Lambda))\\
&=& t K_{\infty}(t^{-1},f, \tilde{l}^{s_1, q_1}_{p_1}(\Lambda), \tilde{l}^{s_0,q_0}_{p_0}(\Lambda))\\
&=& t K_{\infty} (s, f, X_1, X_0)^{\frac{1}{q_1}}.
\end{array}$$
Further,
$$\begin{array}{rcl}
K_{\infty} (s, f, X_1, X_0)= s K_{\infty} (s^{-1}, f, X_0, X_1)&\sim & s \sum\limits_{j\in \mathbb{N}} K_{\infty}(s^{-1} 2^{j\tilde{s}}, b_{j}, Y^j_0, Y^j_1)
\end{array}$$
and
$$\begin{array}{rcl}
K_{\infty}(s^{-1} 2^{j\tilde{s}}, b_{j}, Y^j_0, Y^j_1) &\sim & K_{\infty}(\tau,b_{j}, l^{p_0}(\Gamma_j), l^{p_1}(\Gamma_j))^{q_0}.
\end{array}$$
We get (ii).

\end{proof}


\begin{remark}
Besoy-Haroske-Triebel \cite{BHT} have considered Peetre's conjecture for partial indices.
They have used  trace theorem, wavelets
and some results in \cite{BCT, YCP} to describe the relative wavelet characterization of
$(\dot{B}_{p_0}^{s_0, q_{0}}, \dot{B}_{p_1}^{s_1, q_{1}})_{\theta, r}$ where
$p_0=q_0$, $p_1=q_1, \alpha\in \mathbb{R}$  and $s_0-s_1= \frac{\alpha}{p_0}-\frac{\alpha}{p_1}$.
With the help of new techniques developed in this paper, {\bf we have solved Peetre's conjecture completely.}
Our main techniques are nonlinearity of functional and nonlinearity of lattice topology.
The nonlinearity of functional includes wavelet functional, cuboid functional and vertex functional.
Topological nonlinearity includes topological compatibility of intermediate spaces
and two nonlinear topologies generated by power spaces.

\end{remark}


\textbf{Acknowledgements.} This project is partially supported by the National Natural Science Foundation of China (Grant No. 12071229).
Authors of this paper would like to thank Professor Hans Triebel for his useful discussion, his valuable suggestions, and for providing us with some references.
Further, the final version of the full text has been discussed in detail in the Harmonic Analysis Group of the University of Chinese Academy of Sciences.
Professor Yang would like to Professor Yan's invitation and his hospitality.
The authors would like to thank all members of the group for their valuable opinions,
especially for some improvements to the proof details of the manuscript from Professor Dunyan Yan and Doctor Boning Di.

\bigskip

\noindent Qixiang Yang

\medskip

\noindent School of Mathematics and Statistics, Wuhan University. \\
Wuhan, 430072, China
\smallskip

\noindent{\it E-mail address}:
\texttt{qxyang@whu.edu.cn}

\bigskip

\noindent Haibo Yang

\medskip

\noindent
Macau Institute of Systems Engineering, \\
Macau University of Science and Technology, Macau, 999078, China.\\
\smallskip
\noindent{\it E-mail address}:
\texttt{yanghb97@qq.com}



\bigskip

\noindent Bin Zou

\medskip

\noindent
School of Mathematics and Statistics,\\
Hubei Key Laboratory of Applied Mathematics,\\
Hubei University, Wuhan, 430062, China
\smallskip

\noindent{\it E-mail address}:
\texttt{zoubin0502@hubu.edu.cn}

\bigskip

\noindent Jianxun He

\medskip

\noindent Department of basic course teaching,\\
Software engineering institute of Guangzhou, \\
Guangzhou, 510990, China

\smallskip

\noindent{\it E-mail address}: \texttt{hejianxun@gzhu.edu.cn}
\bigskip

\end{document}